\documentclass{article}

\usepackage[utf8]{inputenc} 
\usepackage{amsmath,amsthm,amssymb}
\usepackage{cases}
\usepackage{mathrsfs} 
\usepackage{ascmac} 
\usepackage{fancybox} 
\usepackage{ulem} 
\usepackage{authblk} 
\usepackage[bbgreekl]{mathbbol}   
\usepackage{physics}  
\usepackage{bm} 
\usepackage{comment}
\usepackage{dashbox} 
\usepackage{csquotes} 

\usepackage[dvipdfmx]{graphicx}
\usepackage[dvipdfmx]{xcolor}
\usepackage{tikz}

\def\rnum#1{\expandafter{\romannumeral #1}}
\def\Rnum#1{\uppercase\expandafter{\romannumeral #1}} 

\theoremstyle{plain}
\newtheorem{thm}{Theorem}[section]

\newtheorem{cor}{Corollary}[section]
\newtheorem{lem}{Lemma}[section]
\newtheorem{rem}{Remark}[section]
\newtheorem{Def}{Definition}[section]

\newtheorem*{thmA}{Theorem A}

\newtheorem*{proA}{Proposition A}

\newcommand{\Rn}{\mathbb{R^{\textit{n}}}}

\newcommand{\N}{\mathbb{N}}
\newcommand{\Z}{\mathbb{Z}}
\newcommand{\R}{\mathbb{R}}

\newcommand{\dis}{\displaystyle}

\newcommand{\vep}{\varepsilon}

\newcommand{\lan}{\langle}
\newcommand{\ran}{\rangle}

\newcommand{\Sp}{\mathcal{S}}
\newcommand{\Lam}{\Lambda}

\newcommand{\calM}{\mathcal{M}}

\newcommand{\what}{\widehat}

\newcommand{\ho}{\"} 
\newcommand{\wti}{\widetilde}

\newcommand{\tn}{\textnormal}

\newcommand{\nn}[1]{{\left\vert\kern-0.25ex\left\vert\kern-0.25ex\left\vert #1 
    \right\vert\kern-0.25ex\right\vert\kern-0.25ex\right\vert}} 

\begin{document}

\title{Two-weight inequality for the heat flow and solvability of Hardy-H\'enon parabolic equation}
\author{Yohei Tsutsui \footnote{Department of Mathematics, Graduate School of Science, Kyoto University, Kyoto 606-8502, Japan. \newline e-mail: \texttt{ytsutsui@math.kyoto-u.ac.jp}}}
\date{}
\maketitle

\begin{abstract}
In this article, we provide two-weight inequalities for the heat flow on the whole space by applying the sparse domination.
For power weights, such inequalities were given by several authors.
Owing to the sparse domination, we can treat general weights in Muckenhoupt classes.
As a application, we present the local and global existence results for the Hardy-H\'enon parabolic equation, which has a potential belonging to a Muckenhoupt class.
\end{abstract}

{\bf Keywords} ,\ Two-weight inequality,\ Sparse domination,\ heat flow,\ Hardy-H\'enon parabolic equation.


\section{Introduction} \label{intro}
The purpose of this paper is twofold.
The first is to provide two-weight inequalities for the heat flow;
\begin{equation} \label{two heat}
\left\| e^{t\Delta}f \right\|_{L^q(w)} \lesssim t^{-\gamma/2} \|f\|_{L^p(\sigma)}
\end{equation}
where the weights $\sigma$ and $w$ belong to Muckenhoupt classes.
The smoothing estimate \eqref{two heat} with power weights $\sigma$ and $w$ was proved by several authors, see \cite{T, O-T, B-T-W, C-I-T2, T-W, C-I-T-T} for example.
Their method is based on the homogeneity of the weighted function spaces, and thus (\ref{two heat}) was reduced to the inequality for the convolution operator with Gaussian.
We need a different approach for the case where the weights $\sigma$ and $w$ are not of power type.
We focus on the fact that the heat flow is a type of smoothing operator.
This indicates that the heat flow is dominated by the fractional integral operator $I_\alpha$;
\[
I_\gamma f(x) := \int_\Rn \dfrac{f(y)}{|x-y|^{n-\gamma}} dy = \mathcal{F}^{-1} \left[|\xi|^{-\gamma} \what{f}(\xi) \right](x)
\]
where ${\what{}}$ and $\mathcal{F}^{-1}$ are the Fourier transform and its inverse, respectively.
Indeed, we make use of a more useful operator $\Lambda_{\Sp}^\gamma$, so-called \textit{fractional sparse operator}, involving a sparse family $\Sp$.
Then, the desired estimate \eqref{two heat} is reduced to one for $\Lambda_\Sp^\gamma$.
Such estimate was already established by Fackler and Hyt\ho{o}nen \cite{F-H}, see Theorem A in Section 2.

The operator $\Lambda_{\Sp} = \Lambda_{\Sp}^0$ is defined by
\[
\Lambda_{\Sp}f := \sum_{Q \in \Sp} \lan f \ran_Q \chi_Q,
\]
and is called \textit{sparse operator}.
This operator appeared in a series of studies on the $A_2$-conjecture, which was solved by Hyt\ho{o}nen \cite{H Ann}.
The sparse operator is very useful for providing pointwise bounds for several operators such as
\begin{equation} \label{pointwise}
|Tf| \lesssim \sum_{a \in \{0,1,2\}^n} \Lambda_{\Sp^a}f,
\end{equation}
with the sparse families $\{\Sp^a\}_a$ depending on $T$ and $f$.
An important property of the sparse operator is the fact that a sharp weighted estimate can be derived from the pointwise bound \eqref{pointwise}.
Indeed, for Calder\'on-Zygmund operator $T$ \cite{H Ann}, the pointwise bound \eqref{pointwise} holds and then we have that
\[
\|T\|_{L^2(w) \to L^2(w)} \lesssim [w]_{A_2},
\]
where the order $1$ of the $A_2$-constant $[w]_{A_2}$ is optimal.
An extrapolation theorem due to Dragic$\check{\tn{e}}$vi\'c, Grafakos, Pereyra and Petermichl \cite{D-G-P-P} yields the sharp $L^p(w)$-bound with the optimal order $\max(1,1/(p-1))$ of the $A_p$-constant.
It is also known that the Littlewood-Paley square function is dominated by a variant of the sparse operator, see Lerner \cite{L}.
Moreover, for example such as maximal spherical operator, pseudo-differential operators and Bochner-Riesz means, the sparse form domination of the form
\begin{equation} \label{sparse form}
\left| \lan Tf,g \ran \right| \lesssim \sum_{a \in \{0,1,2\}^n} \lan f \ran_{r,Q} \lan g \ran_{s,Q} \ \tn{with}\ \lan f \ran_{p,Q} := \left(\dfrac{1}{|Q|} \int_Q |f|^p dx \right)^{1/p}
\end{equation}
with some $1 \le r,s \le \infty$, holds.
See Benea, Bernicot and Luque \cite{B-B-L}, Beltran and Cladek \cite{B-C} and Lacey \cite{Lacey}, respectively.
The estimate \eqref{sparse form} leads to the $L^p(w)$-boundedness for $T$ for all $p \in (r,s^\prime)$, under some condition on the weight $w$.
The weight condition can be described with Muckenhoupt classes.

\medskip

Our first main result is the $L^p(\sigma)-L^q(w)$ inequality for the heat flow.
The proof consists of a pointwise bound for the heat flow by fractional sparse operators and the mapping property of such sparse operator by Fackler and Hyt\ho{o}nen \cite{F-H}, see Section \ref{heat section}.
\begin{thm} \label{main 1}
\tn{(i)} Let $1<p \le q < \infty,\ n(1/p-1/q) \le \gamma < n$ and $\sigma^{1-p^\prime}, w \in A_\infty$.
Then, we have that for $f \in L^p(\sigma)$ satisfying $\lan |f| \ran_Q \to 0$ as $|Q| \to \infty$,
\[
\sup_{t>0} t^{\gamma/2} \left\| e^{t\Delta}f \right\|_{L^q(w)} \lesssim \left[\sigma^{1-p^\prime}, w \right]_{A^\gamma_{p,q}} \|f\|_{L^p(\sigma)},
\]
where the implicit constant depends on $|\sigma^{1-p^\prime}|_{A_\infty}$ and $|w|_{A_\infty}$.

\medskip

\tn{(ii)} Let $1<p<\infty$ and $\sigma^{1-p^\prime} \in A_\infty$.
We assume that there are $p \le q < \infty$ and $n(1/p-1/q) \le \gamma < n$ such that $\sigma^{-1/p} \in \calM^{n/\ell}_{p^\prime}$ with $\ell=\gamma - n(1/p-1/q)$.
Then, we have that for $f \in L^p(\sigma)$ satisfying $\lan |f| \ran_Q \to 0$ as $|Q| \to \infty$,
\[
\sup_{0<t<\infty} t^{(\gamma + n/q)/2} \left\| e^{t\Delta}f \right\|_{L^\infty} \lesssim \left\| \dfrac{1}{\sigma^{1/p}} \right\|_{\calM^{n/\ell}_{p^\prime}} \|f\|_{L^p(\sigma)},
\]
where the implicit constant depends on $|\sigma^{1-p^\prime}|_{A_\infty}$.
\end{thm}

See Section 2 for the definitions of $A^\alpha_{p,q}$- and $A_\infty$-constants.

\medskip

\begin{rem} \label{rem main 1}
\begin{enumerate}
\item
In the case $\sigma$ and $w$ are power weights, the inequality in (i) was proved in \cite{T, O-T, B-T-W, C-I-T2, T-W, C-I-T-T} even if $q<p$.

\item
Since we do not assume that $\sigma \in A_p$, the condition $f \in L^p(\sigma)$ does not imply the local integrability of $f$.

\item
If $\sigma \in A_p$, then the decay condition on $f$ holds for all $f \in L^p(\sigma)$.
Indeed, since $\sigma \not \in L^1$,\ we have that $\lan |f| \ran_Q \le [\sigma]_{A_p}^{1/p} w(Q)^{-1/p} \|f\|_{L^p(\sigma)} \to 0$ as $|Q| \to \infty$.

\item
In the case $\sigma=w$ in \tn{(i)}, we assume that $\sigma, \sigma^{1-p^\prime} \in A_\infty$, which is equivalent to $\sigma \in A_p$.
\end{enumerate}
\end{rem}

\medskip

Most typical examples of weights in Muckenhoupt class $A_p$ is $|x|^\alpha$ and $\lan x \ran^\alpha$.
We give sufficient conditions on such $\alpha$, which satisfy the conditions in Theorem \ref{main 1}.

\begin{cor} \label{cor}
Let $1 < p \le q < \infty,\ \alpha, \beta \in \R$ and $n(1/p-1/q) \le \gamma < n$.
We have that
\[
\sup_{t>0} t^{\gamma/2} \left\| e^{t\Delta}f \right\|_{L^q(w)} \lesssim \|f\|_{L^p(\sigma)}
\]
if one of the following holds

\tn{(i)} $\sigma(x)=|x|^{\alpha p},\ w(x)= |x|^{\beta q}$,
\[
-\dfrac{n}{q} < \beta \le \alpha < n\left(1-\dfrac{1}{p} \right) \ \tn{and} \ \gamma = n \left(\dfrac{1}{p} - \dfrac{1}{q} \right) + (\alpha - \beta).
\]

\tn{(ii)} $\sigma(x)=\lan x \ran^{\alpha p},\ w(x) = |x|^{\beta q}$,
\[
- \dfrac{n}{q} < \beta \le 0 \le \alpha \le n \left(1-\dfrac{1}{p} \right) \ \tn{and}\ n \left(\dfrac{1}{p} - \dfrac{1}{q} \right) -\beta \le \gamma \le \left(\dfrac{1}{p} - \dfrac{1}{q} \right) + (\alpha-\beta).
\]

\tn{(iii)} $\sigma(x) = |x|^{\alpha p},\ w(x) = \lan x \ran^{\beta q}$,
\[
- \dfrac{n}{q} \le \beta \le 0 \le \alpha < n \left(1-\dfrac{1}{p} \right) \ \tn{and} \ n\left(\dfrac{1}{p} - \dfrac{1}{q} \right) + \alpha \le \gamma \le n\left(\dfrac{1}{p} - \dfrac{1}{q} \right) + (\alpha - \beta).
\]

\tn{(iv)} $\sigma(x)=\lan x \ran^{\alpha p},\ w(x) = \lan x \ran^{\beta q}$,
\[
- \dfrac{n}{q} \le \beta \le \alpha \le n\left(1-\dfrac{1}{p} \right)\ \tn{and} \ n \left(\dfrac{1}{p} - \dfrac{1}{q} \right) \le \gamma \le n \left(\dfrac{1}{p} - \dfrac{1}{q} \right) + (\alpha - \beta).
\]
\end{cor}

\bigskip

Our second purpose is to provide existence results for local and global solutions to the Cauchy problem for the Hardy-H\'enon parabolic equation in $\Rn,\ (n \ge 1)$;
\[
\tn{(HH)}
\begin{cases}
\; \partial_t u - \Delta u = V u^\tau,\ & (t>0,\ x \in \Rn) \\
\; u(0) = u_0,\ & (x \in \Rn)
\end{cases}
\]
with a general potential $V$ satisfying a certain Muckenhoupt type condition by applying Theorem \ref{main 1}.
Our results, Theorems \ref{local} and \ref{global}, remain valid if $Vu^\tau$ is replaced by $V|u|^{\tau-1}u$.
The stationary equation with $V(x)= |x|^\gamma$,
\[
(\tn{E})_\alpha \quad -\Delta u = |x|^\gamma u^\tau
\]
was proposed by H\'enon \cite{H} as a model to study rotating stellar systems in astrophysics.

The elliptic equation $(\tn{E})_0$ is known as the Lane-Emden equation.
The existence and non-existence of classical solutions for the equation $(\tn{E})_0$ are well-known.
Gidas and Spruck \cite{G-S} and Chen and Li \cite{C-L} showed that the equation $(\tn{E})_0$ has no classical solutions if $\tau \in (1,\tau_S(0))$, where
\[
\tau_S(\gamma):=
\begin{cases}
\; \dfrac{n + 2 + 2 \gamma}{n-2} \ & \tn{if} \ n >2 \\
\; \infty \ & \tn{if}\ n \le 2.
\end{cases}
\]
On the other hand, the equation $(\tn{E})_0$ allows classical solutions if $\tau \ge \tau_S(0)$, see Gidas and Spruck \cite{G-S} and Liu, Guo and Zhang \cite{L-G-Z}.
For the case $\gamma \not = 0$, Ni \cite{Ni1, Ni2} provided the existence results of classical solutions to $(\tn{E})_\gamma$ when $\tau \ge \tau_S(\gamma)$ and $\gamma > -2$.
The non-existence results for $(\tn{E})_\gamma$ in the case $\tau \in (1,\tau_S(\gamma))$ and $\gamma >-2$ were established by Reichel and Zou \cite{R-Z}.
It is known that the condition $\gamma >-2$ is necessary for the existence of solutions to $(\tn{E})_\gamma$, see Brezis and Cabr\'e \cite{Bre-C}, Mitidieri and Pohozaev \cite{M-P} and Dancer, Du and Guo \cite{D-D-G}.
Recently, Dai and Qin \cite{D-Q} and Giga and Ng$\hat{\tn{o}}$ \cite{G-N} treated the remaining cases, completing the picture on this issue.

While there are several works addressing the existence of solutions to (HH) with $V(x)=|x|^\gamma$ in the case $\gamma <0$, there appears to be limited literature considering the case $\gamma >0$.
For the former case, Ben Slimena, Tayachi and Weissler \cite{B-T-W} and Chikami, Ikeda and Taniguchi \cite{C-I-T1} provided global existence results under the assumption, $-\min (n,2) < \gamma < 0$.
For the latter case, Wang \cite{W} and Hirose \cite{Hirose} constructed global solutions for the case $\gamma >-2$.
The authors showed the existence of the self-similar solutions $u(t,x) = t^{-\beta} U(x/\sqrt{t})$ with some $\beta$, where $U$ is a solution to the corresponding elliptic equation
\[
\Delta U + \dfrac{1}{2} x \cdot \nabla U + \beta U + |x|^\gamma U^p=0.
\]
To the best of the author's knowledge, the only paper that treated a general potential $V$ is Ben Slimena, Tayachi and Weissler \cite[Section 7]{B-T-W}, where they studied the asymptotic behavior of solutions, assuming their existence.

Our results, Theorems \ref{local} and \ref{global}, can treat the case $\gamma =-2$; see also Remarks \ref{rem cor local} and \ref{rem cor global}.
Very recently, Bhimani, Haque and Ikeda \cite{B-H-I} succeeded in constructing local solutions to the linear heat equation with a inverse square potential
\[
\partial_t u - \Delta u + |x|^{-2} u =0
\]
by applying pointwise estimates of the kernel of the heat semigroup $e^{t(-\Delta + |x|^{-2})}$ due to Metafune, Negro and Spina \cite{M-N-S}.
There are many papers concerning the heat equation with a singular potential.
This issue can be traced back to Baras and Goldstein \cite{B-G}, in there the authors provided that the existence of solutions to the heat equation $\partial_t u - \Delta u = c|x|^{-2}u$ under an explicit condition: $0 \le c \le ((n-2)/2)^2$, and the non-existence of solutions when $c>((n-2)/2)^2$.

\medskip

We concern the existence of mild solutions to (HH); $u$ is a mild solution to (HH) if $u$ satisfies
\[
u(t) = e^{t\Delta}u_0 + \int_0^t e^{(t-s)\Delta} \left(Vu(s)^\tau \right) ds.
\]
Our local existence result does not need any auxiliary spaces.
\begin{thm} \label{local}
Let $1<p<\infty$ and $1 \le \tau < \min(p,1+2p/n)$.
We assume that $\sigma \in A_p$ and $\sigma V^{-p/\tau} \in A_{p/\tau}$.
Further, we assume that there is $n(\tau -1)/p \le \beta \le n(1-1/p)$ so that $\beta < 2$ and
\begin{equation} \label{cond local}
W:= \left[\sigma^{-\tau/(p-\tau)} V^{p/(p-\tau)},\ \sigma \right]_{A^\beta_{p/\tau,p}} < \infty.
\end{equation}
Then, for any $u_0 \in L^p(\sigma)$ there exist $T \in (0,\infty)$ and a unique local mild solution $u \in C([0,T); L^p(\sigma))$ to \tn{(HH)}.
\end{thm}

\medskip

\begin{rem} \label{local rem}
\begin{enumerate}
\item
In the case $\sigma \equiv 1$, the condition \eqref{cond local} coincides with a Morrey norm;
\[
W = \|V\|_{\calM^{n/\ell}_{p/(p-\tau)}} \ \tn{with}\ \ell= \beta - \dfrac{n(\tau-1)}{p} \in \left[0, n\left(1 - \dfrac{\tau}{p} \right) \right].
\]
The assumptions on $\beta$ ensure that $\ell \ge 0$ and $p/(p-\tau) \le n/\ell$. 
For the latter condition, it is known that $\calM^p_q = \{0\}$, if $q \le p$ does not hold.
The potential $V(x) = |x|^{-\ell}$ is applicable to Theorem \ref{local}.

\item
If we replace the non-linear term by $V|u|^{\tau-1}u$, then Theorem \ref{local} remains valid.
\end{enumerate}
\end{rem}

\medskip

By using Corollary \ref{cor}, we consider the case that both $\sigma$ and $w$ are power weights $|x|^\alpha$ or non-homogeneous weights $\lan x \ran^\alpha$.
\begin{cor} \label{cor local}
Let $1<p<\infty,\ 1 \le \tau < \min(p,1+2p/n)$.
If one of the following holds, then the statement in Theorem \ref{local} holds.

\tn{(i)} $\sigma(x) = |x|^\alpha,\ V(x) = |x|^\gamma$ and
\[
\begin{cases}
\; \circ \ -n < \alpha < n(p-\tau), \\
\; \circ \ \dfrac{\alpha}{p}(\tau-1) - n \left(1- \dfrac{\tau}{p} \right) \le \gamma \le \dfrac{\alpha}{p}(\tau-1), \\
\; \circ \ \max \left(\dfrac{n+\alpha}{p}\tau - n, \dfrac{n+\alpha}{p}(\tau-1) -2 \right) < \gamma.
\end{cases}
\]

\medskip

\tn{(ii)} $\sigma(x) = \lan x \ran^\alpha,\ V(x) = |x|^\gamma$ and
\[
\begin{cases}
\circ \ 0 \le \alpha \le n(p-1), \\
\circ \ \dfrac{n}{p}\tau - n < \gamma \le 0, \ & \tn{if}\ \alpha=0 \\
\circ \ \dfrac{n+\alpha}{p}\tau - n \le \gamma \le 0, \ & \tn{if}\ \alpha >0 \\
\circ \ \dfrac{n}{p} (\tau-1) - 2 < \gamma.
\end{cases}
\]

\medskip

\tn{(iii)} $\sigma(x) = |x|^\alpha,\ V(x) = \lan x \ran^\gamma$ and
\[
\begin{cases}
\circ \ -n < \alpha <
\begin{cases}
\; n \left(\dfrac{p}{\tau}-1 \right) \ & \tn{if}\ \tau=1 \\
\; \min \left(n \left(\dfrac{p}{\tau} -1 \right), \dfrac{2p}{\tau-1} -n \right) \ & \tn{if}\ \tau>1
\end{cases}, \\
\circ \dfrac{n+\alpha}{p}\tau - n \le \gamma \le 0 \le \dfrac{\alpha}{p}(\tau-1) \ ,
\end{cases}
\]

\medskip

\tn{(iv)} $\sigma(x) = \lan x \ran^\alpha,\ V(x) = \lan x \ran^\gamma$ and
\[
\begin{cases}
\circ \ - n \le \alpha \le n(p-\tau), \\
\circ \ \max \left(\dfrac{n+\alpha}{p}\tau - n,\ \dfrac{\alpha}{p}(\tau-1)-n \left(1-\dfrac{\tau}{p} \right) \right) \le \gamma \le \min \left(\dfrac{n+\alpha}{p}\tau, \dfrac{\alpha}{p} (\tau-1) \right), \\
\; \circ \ \dfrac{n+\alpha}{p}(\tau-1) -2 < \gamma.
\end{cases}
\]
\end{cor}

\medskip

\begin{rem} \label{rem cor local}
\begin{enumerate}
\item
As mentioned in Introduction, in the previous papers, it has been difficult to study the case $\gamma \le -2$.
However, our second case (ii) includes the case $\gamma = -2$, for example.
Indeed, if we assume that
\[
n \ge 3,\ n/(n-2) \le p \ \tn{and}\ 1 < \tau < \min \left(p,\ 1 + \dfrac{2p}{n},\ p \left(1 - \dfrac{n}{2} \right) \right)
\]
then $\gamma=-2$ is acceptable, while the first case (i) does not include such a case.

\item
If $\alpha > 0$ and $\tau >1$, then the cases (i) and (iv) can treat a positive $\gamma$.
\end{enumerate}
\end{rem}

\medskip

Our global result can read as follows.
\begin{thm} \label{global}
Let $\max(1,n/2) < p < \infty,\ 2 < \tau < 1+2p/n,\ \sigma \in A_p,$ and $\sigma V^{-p} \in A_p$.
We assume that there exist $p \le q,q_1<\infty$ and $n(1/p-1/q) \le \alpha \le \min(n,2/(\tau-1)-n/q)$ such that
\[
(\ast)
\begin{cases}
\; 2-n\left(1+\dfrac{1}{{q_1}} \right) < \left(\alpha + \dfrac{n}{q} \right)(\tau-2), \\
\; W_1:= \left[\sigma^{1-p^\prime} V^{p^\prime}, \sigma \right]_{A^\beta_{p,p}} < \infty, & \tn{with}\ \beta:= 2 - \left(\alpha + \dfrac{n}{q} \right)(\tau-1) \\
\; W_2 := \left[\sigma^{1-p^\prime},1 \right]_{A^\alpha_{p,q}} = \left\| \dfrac{1}{\sigma^{1/p}} \right\|_{\calM^{n/\ell}_{p^\prime}} < \infty, \ & \tn{with}\ \ell := \alpha - n \left(\dfrac{1}{p} - \dfrac{1}{q} \right) \\
\; W_3:= \left[\sigma^{1-p^\prime} V^{p^\prime}, 1 \right]_{A^{\beta_1}_{p,q_1}} = \left\| \dfrac{V}{\sigma^{1/p}} \right\|_{\calM^{n/{\ell_1}}_{p^\prime}} < \infty,\ & \tn{with}\ \ell_1 := 2 - \left(\alpha + \dfrac{n}{q} \right)(\tau-2) - \dfrac{n}{p},
\end{cases}
\]
where $\beta_1= \alpha +\beta+n(1/q-1/{q_1})$.
Then, for any small $u_0 \in L^p(\sigma)$ there is a global mild solution $u \in C([0,\infty); L^p(\sigma))$ to \tn{(HH)}, which satisfies
\[
\sup_{t>0} t^{(\alpha+n/q)/2} \|u(t)\|_{L^\infty} \lesssim \|u_0\|_{L^p(\sigma)}.
\]
\end{thm}

\medskip

\begin{rem} \label{global rem}
\begin{enumerate}
\item
The first condition in $(\ast)$ holds whenever $n \ge 2$.

\item
In the condition on $W_3$, the exponent $q_1$ is apparent, because the Morrey norm does not include $q_1$.
Thus, combining with the first remark, the exponent $q_1$ is relevant only in the case $n=1$.

\item
The constant $W_1$ is also related to a Morrey type norm;
\[
W_1 \le [\sigma]_{A_p}^{1/p} \sup_Q |Q|^{\beta/n - (1/p-1/q)} \left(\dfrac{1}{\sigma^{1-p^\prime}(Q)} \int_Q V^{p^\prime} \sigma^{1-p^\prime} dx \right)^{1/{p^\prime}}.
\]
In the case $\sigma \equiv 1$, the equality holds.

\item
If $n \ge 2,\ \sigma \equiv 1,\ q=p$ and $\alpha=0$, it follows that
\[
W_1=W_3= \|V\|_{\calM^{n/\beta}_{p^\prime}} \ \tn{and}\ W_2=1.
\]
Therefore, the potential $V(x)=|x|^{-\beta} = |x|^{-(2-n(\tau-1)/p)} \in \calM^{n/\beta}_{p^\prime}$ is applicable to Theorem \ref{global}.
Here, $\beta \in (0,2)$.

\item
It is known that if $\alpha < n(1/p-1/q)$ and $[\sigma,w]_{A^\alpha_{p,q}} < \infty$, then one of $\sigma$ and $w$ is $0$.
See \cite{F-H} for this fact.


\item
Theorem \ref{global} remains valid if we replace the non-linear term with $V|u|^{\tau-1}u$.
\end{enumerate}
\end{rem}

\medskip

Corollary \ref{cor} provides sufficient conditions on $\sigma$ and $V$, when both are a power weight or a non-homogeneous weight.
\begin{cor} \label{cor global}
Let $\max(1,n/2) < p < \infty,\ 2 < \tau < 1+2p/n$.
If one of the following holds, then Theorem \ref{global} holds.

\tn{(i)} $\sigma(x) = |x|^\alpha,\ V(x) = |x|^\gamma$ with arbitrary $q=q_1 \in [p,\infty)$ and
\[
\begin{cases}
\; \circ \ 0 \le \alpha < \min \left(n(p-1),\ \dfrac{2p}{\tau-1} -n \right), \\
\; \circ \ - \dfrac{n-2}{\tau-2}p -n < \alpha, \\
\; \circ \ \gamma = \dfrac{n+\alpha}{p} (\tau-1) - 2 (< 0).
\end{cases}
\]

\medskip

\tn{(ii)} $\sigma(x) = \lan x \ran^\alpha,\ V(x) = |x|^\gamma$ with arbitrary $q = q_1 \in [p,\infty)$ and
\[
\begin{cases}
\; \circ \ 0 \le \alpha \le n(p-1) \\
\; \circ \ \max \left(\dfrac{n+\alpha}{p} - n,\ -n\left(1-\dfrac{1}{p} \right)(\tau-1)-2 \right) \le \gamma \le \dfrac{n+\alpha}{p}(\tau-1)-2,\\
\; \circ \ 1 + \dfrac{2}{n} \le \tau.
\end{cases}
\]

\medskip

\tn{(iii)} $\sigma(x) = |x|^\alpha,\ V(x) = \lan x \ran^\gamma$ with arbitrary $q = q_1 \in [p,\infty)$ and
\[
\begin{cases}
\; \circ \ n \ge 2,\\
\; \circ \ 0<\alpha < n(p-1), \\
\; \circ \ \alpha \le 2p-n(\tau-1), \\
\; \circ \ \dfrac{n+\alpha}{p} -n \le \gamma \le \dfrac{\alpha/p-2}{\tau-1} + \dfrac{n}{p} (\le 0),\\
\; \circ \ -n\left(1-\dfrac{1}{p} \right) < \gamma.
\end{cases}
\]

\medskip

\tn{(iv)} $\sigma(x) = \lan x \ran^\alpha,\ V(x) = \lan x \ran^\gamma$ with $q=q_1=p$ and
\[
\begin{cases}
\; \circ \ n =2,3 \ \tn{or}\ n \ge 4 \ \tn{with} \ \dfrac{n}{n-2} \le p, \\
\; \circ \ 0 \le \alpha \le \min \left(n(p-1), \dfrac{2p-n}{\tau-2}+p(n-2)-n,\ \dfrac{2p}{\tau-1} + p(n-2) \right), \\
\; \circ \ \alpha \le \dfrac{n(p+1)-2p}{3-\tau} \ \tn{if}\ 2<\tau<3, \\
\; \circ \ \dfrac{n+\alpha}{p} -n \le \gamma \le \min \left(\dfrac{n+\alpha}{p}(\tau-1)-2,\ \dfrac{2-\alpha/p}{\tau-2} -2,\ \dfrac{2}{\tau-1} + \dfrac{n}{p} -2,\ \dfrac{n+\alpha}{p}(\tau-2) + \dfrac{n}{p} - 2 \right), \\
\; \circ \ \gamma < \dfrac{n}{p} - \dfrac{2}{\tau-2} (<0).
\end{cases}
\]
\end{cor}

\medskip

\begin{rem} \label{rem cor global}
\begin{enumerate}
\item
The second case (ii) only allows us to take a positive $\gamma$.
In fact, if we assume that
\[
\max \left(1, \dfrac{n}{2} \right) < p< \infty,\ 0 \le \alpha \le n(p-1),\ 1+\dfrac{2}{n} < \tau < 1 + \dfrac{2p}{n},
\]
then we can take a positive $\gamma$.

\item
The second case (ii) also includes the case $V(x)=|x|^{-2}$, if we assume that
\[
n \ge 3,\ \dfrac{n}{2} < p < \infty,\ \dfrac{n}{n-2} \le p,\ 2 < \tau < 1 + \dfrac{2p}{n} \ \tn{and}\ 0 \le \alpha \le p\left(n-2 \right)-n.
\]
\end{enumerate}
\end{rem}

\medskip

This paper is organized as follows.
In Section 2, we recall the definitions of Morrey spaces and weight classes, and also basic facts about dyadic analysis.
In Section 3, we provide a sparse domination of the heat flow, and then prove Theorem \ref{main 1}.
We prove the existence of local and global solutions in Section 4.
Finally, in Section 5, we give a sketch of proofs of Corollaries \ref{cor}, \ref{cor local} and \ref{cor global}.

\vspace{5mm}

\section{Preliminary} \label{preliminary}
Throughout this paper, we use the following notation.
By a \textit{cube}, we mean a cube $Q$ in $\Rn$ with sides parallel to the coordinate axes.
Let $\ell(Q)$ denote its side length, and $|Q|$ its Lebesgue measure.
We call a Lebesgue measurable function $w$ on $\Rn$ a \textit{weight} if $w$ is a non-negative and locally integrable function.
For a weight $w$, a cube $Q$ and $1 \le p < \infty$, we denote
\[
w(Q):= \int_Q w dx,\ \lan f \ran_Q := \dfrac{1}{|Q|} \int_Q f dx \ \tn{and}\ \|f\|_{L^p(\sigma)} := \|f \sigma^{1/p}\|_{L^p}.
\]
We denote $\chi_E$ by the characteristic function of a measurable subset $E \subset \Rn$.
The notation $A \lesssim B$ denotes $A \le c B$ for a positive constant $c$ independent of $A$ and $B$, and also $A \approx B$ means that $A \lesssim B$ and $A \gtrsim B$.
Let us denote $C^\infty_0$ as the space of smooth functions having compact support.

\medskip

In this section, we recall the definitions of Morrey spaces, weight classes and some notations of dyadic analysis including the fractional sparse operators.

\bigskip

\subsection{Morrey spaces}
Morrey space $\calM^p_q$ is applied as a class for weights in the present paper.

\begin{Def} \label{Morrey def}
Let $1 \le q \le p \le \infty$.
\[
\|f\|_{\calM^p_q} := \sup_Q |Q|^{1/p-1/q} \|f\|_{L^q(Q)}.
\]
\end{Def}

\medskip

\begin{rem} \label{Morrey rem}
\begin{enumerate}
\item
If $p<q$ holds, then $\calM^p_q = \{0\}$.

\item
\[
L^p \hookrightarrow L^{p,\infty} \hookrightarrow \calM^p_q,
\]
where $L^{p,\infty}$ is the weak $L^p$ space.
\end{enumerate}
\end{rem}

\medskip

Morrey spaces appear in Fefferman-Phong inequality \cite{F}; if $n \ge 3$ and $2<q \le n$, it holds that
\begin{equation} \label{Fefferman-Phong}
\|Vf\|_{L^2} \lesssim \|V\|_{\calM^n_q} \|\nabla f\|_{L^2}.
\end{equation}
These inequalities imply that the Schr\ho{o}dinger operator $-\Delta + V$ is positive definite and is therefore a self-adjoint operator on $L^2$, provided that $\|V\|_{\calM^n_q}$ is small.
The inequality \eqref{Fefferman-Phong} with $L^n$ in the place of $\calM^n_q$ follows from a combination of H\ho{o}lder's inequality and the Sobolev inequality.
One of the advantages of \eqref{Fefferman-Phong} is the fact that the Coulomb potential $|x|^{-1}$ belongs to $\calM^n_q$ but not to $L^n$.
The Fefferman-Phong inequality \eqref{Fefferman-Phong} can be regarded as a special case of two-weight inequalities for the fractional integral operator $I_1$.
Morrey spaces are implicitly present in our estimates for the heat flow; see Theorem \ref{main 1}

\bigskip

\subsection{Weight classes}
Muckenhoupt classes $A_p,\ (1 \le p \le \infty),$ are suitable for the $L^p$ boundedness of several operators, for example, Hardy-Littlewood maximal operator $M$ and Calder\'on-Zygmund operators.

\begin{Def} \label{A_p}
\tn{(i)} Let $1 \le p \le \infty$ and $w$ be a weight.
\[
[w]_{A_p} :=
\begin{cases}
\; \dis \sup_Q \lan w \ran_Q \|w^{-1}\|_{L^\infty(Q)} \ & \tn{if}\ p=1 \\
\; \dis \sup_Q \lan w \ran_Q \lan w^{1-p^\prime} \ran_Q^{p-1} \ & \tn{if}\ p \in (1,\infty),\\
\; \dis \sup_Q \lan w \ran_Q \exp \lan \log w^{-1} \ran_Q,\ & \tn{if}\ p=\infty.
\end{cases}
\]

\tn{(ii)} [Fujii \cite{Fujii}, Wilson \cite{Wilson}]
For a weight $w$,
\[
|w|_{A_\infty} := \sup_Q \dfrac{1}{w(Q)} \int_Q M(\chi_Q w) dx.
\]
\end{Def}

\medskip

\begin{rem} \label{A_p rem}
\begin{enumerate}
\item
If $1 \le p < \infty$ and $w \in A_p$, then $C^\infty_0$ is dense in $L^p(w)$.
See Nakai, Tomita and Yabuta \cite{N-T-Y} for the proof. 

\item
Hyt\ho{o}nen and P\'erez \cite{H-P} demonstrated that $[w]_{A_\infty} < \infty \iff |w|_{A_\infty} < \infty$, and the bound $|w|_{A_\infty} \lesssim [w]_{A_\infty}$.
It is well-known that the reverse inequality fails in general.

\item
We make use of the properties of $A_p$ weights; for $1<p<\infty$
\[
w \in A_p \iff w, w^{1-p^\prime} \in A_\infty,
\]
and for $f \in L^p(w)$ and $w \in A_p$ with $1<p<\infty$,
\[
\lan |f| \ran_Q \le [w]_{A_p}^{1/p} \dfrac{1}{w(Q)^{1/p}} \|f\|_{L^p(w)} \to 0 \ \tn{as}\ |Q| \to \infty
\]
in Section \ref{equation section}.
This convergence follows from the fact that $w \not \in L^1$.
\end{enumerate}
\end{rem}

\medskip

The $A_{p,q}^\alpha$ classes appear as a condition for $L^p-L^q$ bounds for the fractional operators.

\begin{Def} \label{two weight class}
Let $1 \le p,q \le \infty,\ \alpha \in \R$ satisfy $n(1/p-1/q) \le \alpha$, and $\sigma$ and $w$ are weights.
We say $(\sigma,w) \in A^\alpha_{p,q}$ if
\[
[\sigma, w]_{A^\alpha_{p,q}} := \sup_Q |Q|^{\alpha/n - (1/p-1/q)} \lan \sigma \ran_Q^{1/{p^\prime}} \lan w \ran_Q^{1/q} < \infty.
\]
\end{Def}

\medskip

It is known that the $A_{p,q}^\alpha$-condition characterizes the $L^p(\sigma)-L^{q,\infty}(w)$ bound for the fractional maximal operator $M^\alpha$;
\[
M^\alpha f := \sup_Q |Q|^{\alpha/n} \lan |f| \ran_Q \chi_Q.
\]
See the lecture note by Cruz-Uribe \cite[p.60]{CU}.
However, the $A_{p,q}^\alpha$-condition is not sufficient for the $L^p(\sigma)-L^q(w)$ bound for $M^\alpha$.
Counterexamples were given by Muckenhoupt and Wheeden \cite{M-W} for $\alpha=0$, and Cruz-Uribe \cite{CU} for $0<\alpha < n$.
The $A_{p,q}^\alpha$-condition is also not sufficient for the $L^p(\sigma)-L^{q,\infty}(w)$ bound for $I_\alpha$.
See Cruz-Uribe \cite[p.63]{CU} for the counterexample.
However, Fackler and Hyt\ho{o}nen \cite{F-H} showed that the $A_{p,q}^\alpha$-condition and the $A_\infty$-conditions for both $\sigma$ and $w$ ensure the $L^p(\sigma)-L^q(w)$ bound for $I_\alpha$.

\medskip

\begin{rem} \label{two weight class rem}
\begin{enumerate}
\item
Fackler and Hyt\ho{o}nen \cite{F-H} demonstrated that if $\alpha < n(1/p-1/q)$ and $[\sigma,w]_{A^\alpha_{p,q}} < \infty$, then $\sigma \equiv 0$ or $w \equiv 0$.
Morrey spaces have a similar property, see Remark \ref{Morrey rem}.

\item
A relationship between $A_{p,q}^\alpha$-classes and Morrey spaces is as follows: in the case $w \equiv 1$, one sees
\[
[\sigma,1]_{A^\alpha_{p,q}} = \left\|\sigma^{1/{p^\prime}} \right\|_{\calM^{n/\ell}_{p^\prime}} \ \tn{with}\ \ell=\alpha - n/{q^\prime}.
\]
Such quantities appear in Theorems \ref{local} and \ref{global}.

\item
Although it holds $[w]_{A_p} \ge 1$, the constant $[\sigma,w]_{A_{p,q}^\alpha}$ can be small.

\item
Let $1<p,q<\infty,\ \alpha, \beta,\gamma \in \R,\ \sigma(x) := \lan x \ran^{\alpha p}$ and $w(x) := \lan x \ran^{\beta q}$.
If
\[
n\left(\dfrac{1}{p} - \dfrac{1}{q} \right) \le \gamma \le n \left(\dfrac{1}{p} - \dfrac{1}{q} \right) + (\alpha - \beta) + \min \left(0,n \left(1-\dfrac{1}{p} \right) -\alpha \right) + \min \left(0,\dfrac{n}{q} + \beta \right),
\]
then we can see that $(\sigma^{1-p^\prime},w) \in A_{p,q}^\gamma$.
Both minimums equal $0$ if $\sigma^{1-p^\prime}, w \in A_\infty$.
This condition appears in Theorems \ref{local} and \ref{global}.
\end{enumerate}
\end{rem}

\bigskip

\subsection{Dyadic analysis}
We recall some notions on dyadic analysis, and results on the mapping properties of the fractional sparse operator $\Lambda_{\Sp}^\gamma$.
The reader can refer to Lerner and Nazarov \cite{L-N} for the former.
Since the latter is our main tool for the proof of Theorem \ref{main 1}, we present the result as Theorem A below.
To state the theorem, let us define basic notation for the sparse operator $\Lambda_{\Sp}^\gamma$.

\begin{Def} \label{sparse def}
\tn{(i)} A collection $D$ of cubes in $\Rn$ is called a dyadic grid if $D$ satisfies the following conditions
\[
\begin{cases}
\; \circ \ \tn{If}\ Q \in D,\ \tn{then the side length}\ \ell(Q) = 2^k \ \tn{for some}\ k \in \Z, \\
\; \circ \ \tn{If}\ Q_1, Q_2 \in D,\ \tn{then}\ Q_1 \cap Q_2 \in \{\emptyset,\ Q_1, Q_2\}, \\
\; \circ \ \tn{For every}\ k \in \Z,\ \tn{it holds that}\ \cup_{Q \in D_k} Q = \Rn \ \tn{where}\ D_k:= \{Q \in D;\ \ell(Q)=2^k\}. 
\end{cases}
\]

\medskip

\tn{(ii)} For any $a \in \{0,1,2\}^n$, the shifted dyadic cubes $\mathcal{D}^a$ are defined as follows;
\[
\mathcal{D}^a := \left\{2^{-j} \left([0,1)^n + m + (-1)^j\dfrac{a}{3} \right); j \in \Z,\ m \in \Z^n \right\}.
\]

\medskip

\tn{(iii)} Let $\eta \in (0,1)$.
A family $\Sp$ of cubes in $\Rn$ is called $\eta$-sparse family if there are disjoint subsets $\{E_Q\}_{Q \in \Sp}$ satisfying $E_Q \subset Q$ and $\eta |Q| \le |E_Q|$.

\medskip

\tn{(iv)} [Fackler and Hyt\ho{o}nen \cite{F-H}] For an $\eta$-sparse family $\Sp,\ \gamma \in [0,n)$ and a locally integral function $f$, we define the fractional sparse operator by
\[
\Lambda_{\Sp}^\gamma f := \sum_{Q \in \Sp} |Q|^{\gamma/n} \lan f \ran_Q \chi_Q.
\]
\end{Def}

\medskip

\begin{rem} \label{sparse def rem}
\begin{enumerate}
\item
Since the index $\eta$ in the definition above is not important in the present paper, we omit it.

\item
The dyadic grid $\mathcal{D}^a$ is not a sparse family.
\end{enumerate}
\end{rem}

\medskip

One of characteristic properties of $\{\mathcal{D}^a\}_{a \in \{0,1,2\}^n}$ is that \textit{every cubes} are dominated from above by cubes in them in the following sense.
\begin{proA} [Hyt\ho{o}nen, Lacey and P\'erez \cite{H-L-P}]
For any cube $Q$ in $\Rn$, there exist $a \in \{0,1,2\}^n$ and a cube $\wti{Q} \in \mathcal{D}^a$ such that $Q \subset \wti{Q}$ and $|Q| \approx \left|\wti{Q} \right|$.
\end{proA}

\medskip

Two-weight inequality for $\Lambda_{\Sp}^\alpha$ was proved by Fackler and Hyt\ho{o}nen \cite{F-H}.
This bound plays an important role in the proof of our main result Theorem \ref{main 1}.
\begin{thmA} [Fackler and Hyt\ho{o}nen \cite{F-H}]
Let $1< p \le q < \infty,\ 0 \le \alpha < n,\ \sigma,w \in A_\infty$ and $(\sigma,w) \in A_{p,q}^\alpha$.
Assume that $\Sp$ is a sparse family in a dyadic grid $\mathcal{D}$.
Then, we have that
\[
1 \le \dfrac{\left\| \Lambda_{\Sp}^\alpha (\cdot \sigma) \right\|_{L^p(\sigma) \to L^q(w)}}{[\sigma,w]_{A_{p,q}^\alpha}} \le C,
\]
where the constant $C$ depends on $n,p,q,\alpha$ and the $A_\infty$-constants of $\sigma$ and $w$.
\end{thmA}

\vspace{5mm}

\section{Proof of two weight inequality for the heat flow} \label{heat section}
In this section, we prove Theorem \ref{main 1}.
One of the keys for the proof is the pointwise bound of the heat flow by the fractional sparse operator;
\begin{equation} \label{heat pointwise}
\left| e^{t\Delta}f(x) \right| \lesssim t^{-\gamma} \sum_{a \in \{0,1,2\}^n} \Lambda_{\Sp^a}^\gamma f(x)
\end{equation}
where $0 \le \gamma < n$, and $\{\Sp^a\}_a$ are sparse families that depend on $f$, but not on $t$.
It suffices to prove this theorem for non-negative functions $f$.
Therefore, in this section, we assume that $f$ is a non-negative function.

\bigskip

\subsection{The pointwise bound of the heat flow by the fractional sparse operators}
We shall verify the pointwise estimate \eqref{heat pointwise}.
Let $f \in L^p(\sigma)$ satisfy $\lan f \ran_Q \to 0$ as $|Q| \to \infty$.
The first half of the proof of \eqref{heat pointwise} is very elementary, while the latter half follows a well-known argument.

For a sufficiently large $N \in \N$, we have that
\begin{align*}
|e^{t\Delta}f(x)| & \lesssim t^{-n/2} \dis \sum_{k \in \Z} \left(\dfrac{1}{1+2^k}\right)^N \int_{|x_j-y_j| \le 2^k \sqrt{t}} |f| dy \\
& \le t^{-n/2} \dis \sum_{k \in \Z} \left(\dfrac{1}{1+2^k}\right)^N \sum_{a \in \{0,1,2\}^n} \sum_{\substack{Q \in \mathcal{D}^a \\ \ell(Q) \approx 2^k \sqrt{t}}} \int_Q |f| dy \chi_Q(x) \\
& \lesssim t^{-\gamma/2} \sum_{a \in \{0,1,2\}^n} \sum_{Q \in \mathcal{D}^a} |Q|^{\gamma/n} \lan f \ran_Q \chi_Q(x) \left(\sum_{k \in \Z} \left(\dfrac{1}{1+2^k} \right)^N 2^{(n-\gamma)k} \right) \\
& \lesssim t^{-\gamma/2} \sum_{a \in \{0,1,2\}^n} \Lam^\gamma_{\mathcal{D}^a}f(x) \\
& \lesssim t^{-\gamma/2} I_\gamma f(x).
\end{align*}
Here, we have used Proposition A in the second inequality, and Proposition 3.3 from \cite{CU} in the last inequality.
Because $\mathcal{D}^a$ is not a sparse family, we have to replace $\{\mathcal{D}^a\}_a$ with some sparse families.
This argument is well-known, but for the sake of completeness, we present it here.
Let us take $\lambda > 2^n$, and for any $a \in \{0,1,2\}^n$ and $k \in \Z$.
Let us denote maximal cubes $\{Q^a_{j,k}\}_{j=1}^\infty \subset \mathcal{D}^a$ satisfying $\lambda^k < \lan f \ran_{Q^a_{j,k}}$.\footnote{Here, we used the decay condition on $f$.}
From the maximality, it holds
\[
\lambda^k < \lan f \ran_{Q^a_{j,k}} \le 2^n \lambda^k
\]
for all $a, j$ and $k$.
Next, let us denote
\[
C_k^a := \{Q \in \mathcal{D}^a; \lambda^k < \lan f \ran_Q \le \lambda^{k+1}\}.
\]
One can observe the following two properties;
\[
\begin{cases}
\; \tn{if}\ Q \in \mathcal{D}^a \ \tn{and}\ \lan f \ran_Q \not = 0,\ \tn{then}\ \tn{there is a unique}\ k \in \Z \ \tn{such that}\ Q \in C_k^a \\
\; \tn{if}\ Q \in C_k^a,\ \tn{then}\ \tn{there is a unique}\ j \in \N\ \tn{such that}\ Q \subset Q^a_{j,k} \ \tn{and}\ \lan f \ran_Q \le \lambda \lan f \ran_{Q^a_{j,k}}.
\end{cases}
\]
By applying these properties, we can see that
\begin{align*}
\Lam^\gamma_{\mathcal{D}^a}f & = \dis \sum_{\substack{Q \in \mathcal{D}^a \\ \lan f \ran_Q \not = 0}} |Q|^{\gamma/n} \lan f \ran_Q \chi_Q \\
& \le \sum_{k \in \Z} \sum_{Q \in C^a_k} |Q|^{\gamma/n} \lan f \ran_Q \chi_Q \\
& \le \sum_{k \in \Z} \sum_{j=1}^\infty \sum_{\substack{Q \in C^a_k \\ Q \subset Q_{j,k}^a}} |Q|^{\gamma/n} \lan f \ran_Q \chi_Q \\
& \le \lambda \sum_{k \in \Z} \sum_{j=1}^\infty \lan f \ran_{Q_{j,k}^a} \sum_{\substack{Q \in C^a_k \\ Q \subset Q^a_{j,k}}} |Q|^{\gamma/n} \chi_Q.
\end{align*}
Here, if $\ell(Q^a_{j,k}) \approx 2^{-\ell}$ with some $\ell \in \Z$, then it holds
\[
\dis \sum_{\substack{Q \in C^a_k \\ Q \subset Q^a_{j,k}}} |Q|^{\gamma/n} \chi_Q \lesssim \sum_{m=\ell}^\infty 2^{-m\gamma} \chi_{Q^a_{j,k}} \lesssim \left|Q^a_{j,k} \right|^{\gamma/n} \chi_{Q^a_{j,k}}.
\]
Taking into account, one obtains
\[
\Lam^\gamma_{\mathcal{D}^a}f \lesssim \dis \sum_{j,k} \left|Q^a_{j,k} \right|^{\gamma/n} \lan f \ran_{Q^a_{j,k}} \chi_{Q^a_{j,k}} = \Lam^\gamma_{\Sp^a}f
\]
with $\Sp^a := \left\{Q^a_{j,k} \right\}_{j,k}$.
Because we can see that
\[
E_{Q^a_{j,k}} := \left\{x \in Q^a_{j,k} : \lambda^k < M_{\mathcal{D}^a} f(x) \le \lambda^{k+1} \right\},\ \text{with}\ M_{\mathcal{D}^a} f := \dis \sup_{Q \in \mathcal{D}^a} \lan f \ran_Q \chi_Q,
\]
enjoys the property that $|E_{Q^a_{j,k}} | \ge \eta |Q^a_{j,k}|$ if $a \gg 1$ and the sets $\{E^a_{j,k} \}_{j,k}$ are pairwise disjoint for each $a \in \{0,1,2\}^n$, the collections $\{\Sp^a \}_{a \in \{0,1,2\}^n}$ are $\eta$-sparse families.
Thus, we have achieved the desired sparse domination
\[
|e^{t\Delta}f| \lesssim t^{-\gamma/2} \dis \sum_{a \in \{0,1,2\}^n} \Lam^\gamma_{\Sp^a}f.
\]
\begin{flushright}
    $\square$
\end{flushright}

\bigskip

\subsection{Proof of Theorem \ref{main 1}}
We are now in a position to prove Theorem \ref{main 1}.

\medskip

(i) From the pointwise bound \eqref{heat pointwise}, it suffices to prove that for a sparse family $\Sp$,
\[
\left\| \Lambda_{\Sp}^\gamma f \right\|_{L^q(w)} \lesssim \left[\sigma^{1-p^\prime}, w \right]_{A^\gamma_{p,q}} \|f\|_{L^p(\sigma)}.
\]
This is immediately shown with the aid of Theorem A; for $g:= f \sigma^{-(1-p^\prime)}$,
\[
\left\| \Lambda_{\Sp}^\gamma f \right\|_{L^q(w)} = \left\| \Lambda_{\Sp}^\gamma (g \sigma^{1-p\prime}) \right\|_{L^q(w)} \lesssim \left[\sigma^{1-p^\prime}, w \right]_{A^\gamma_{p,q}} \left\|g\right\|_{L^p(\sigma^{1-p^\prime})} = \left[\sigma^{1-p^\prime}, w \right]_{A^\gamma_{p,q}} \|f\|_{L^p(\sigma)}.
\]

\medskip

(ii) From the smoothing $L^q-L^\infty$ estimate for $e^{t\Delta/2}$ and (i), one obtains that
\[
\left\| e^{t\Delta}f \right\|_{L^\infty} \lesssim t^{-(\gamma + n/q)/2} \left[\sigma^{1-p^\prime},1 \right]_{A_{p,q}^\gamma} \|f\|_{L^p(\sigma)}.
\]
Since $\ell=\gamma - n(1/p-1/q)$ and
\[
\left[\sigma^{1-p^\prime},1 \right]_{A_{p,q}^\gamma} = \sup_Q |Q|^{\gamma/n - (1/p-1/q)} \lan \sigma^{1-p^\prime} \ran_Q^{1/{p^\prime}} = \sup_Q |Q|^{\gamma/n - (1/p-1/q) - 1/{p^\prime}} \left\|\dfrac{1}{\sigma^{1/p}} \right\|_{L^{p^\prime}(Q)} = \left\| \dfrac{1}{\sigma^{1/p}} \right\|_{\calM^{n/\ell}_{p^\prime}},
\]
the proof is completed.
\begin{flushright}
    $\square$
\end{flushright}

\vspace{5mm}

\section{Proofs of the existence of solutions to Hardy-H\'enon parabolic equation} \label{equation section}
We prove Theorems \ref{local} and \ref{global} by the standard fixed point argument.

\bigskip

\subsection{Proof of Theorem \ref{local}: Local existence}
Let us defnote for $t \in (0,\infty)$,
\[
\|u\|_{X_p(t)} := \sup_{0<s<t} \|u(s)\|_{L^p(\sigma)}
\]
and $\Phi[u](t) := e^{t\Delta}u_0 + D[u](t)$ where 
\[
D[u](t) := \int_0^t e^{(t-s)\Delta} \left(V u(s)^\tau \right) ds = \int_0^t e^{(t-s)\Delta} U(s) ds.
\]
We shall show that $\Phi$ is a contraction mapping on $BC([0,T); L^p(\sigma))$ if $T>0$ is small.

\smallskip

\underline{Step 1:} For any $T \in (0,\infty)$ and $u \in BC([0,T); L^p(\sigma))$, we have
\[
\|\Phi[u]\|_{X_p(T)} \lesssim W T^{1-\gamma/2} \|u\|_{X_p(T)}^\tau.
\]

\smallskip

The $A_p$-condition on $\sigma$ implies that $\|e^{t\Delta} u_0\|_{L^p(\sigma)} \lesssim \|u_0\|_{L^p(\sigma)}$.
By using Theorem \ref{main 1} with $p=r:=p/\tau \in (1,p]$, one obtains that
\begin{equation} \label{loc Lp}
\begin{aligned}
\left\| D[u](t) \right\|_{L^p(\sigma)} & \lesssim \left[\sigma^{1-r^\prime} V^{-r(1-r^\prime)},\sigma \right]_{A_{r,q}^\gamma} \int_0^t (t-s)^{-\gamma/2} \left\| U(s) \right\|_{L^r(\sigma V^{-r})} ds \\
& = W \int_0^t (t-s)^{-\gamma/2} \|u(s)\|_{L^p(\sigma)}^\tau ds \\
& \lesssim t^{1-\gamma/2} W \|u\|_{X_p(t)}^\tau.
\end{aligned}
\end{equation}
The first inequality requires the condition that $\lan |U(s)| \ran_Q \to 0$ as $|Q| \to \infty$.
This condition can be verified from the facts that $U(s) \in L^r(\sigma V^{-r})$ and $\sigma V^{-r} \in A_r$.
See Remark \ref{A_p rem}.
Hence, Step 1 is proved.

\smallskip

\underline{Step 2:} For any $T \in (0,\infty)$ and $u \in BC([0,T);L^p(\sigma))$,
\[
\Phi[u] \in BC([0,T); L^p(\sigma))
\]
satisfying $\Phi[u](t) \to u_0$ in $L^p(\sigma)$ as $t \searrow 0$.

\smallskip

Firstly, our assumption $\sigma \in A_p$ yields that $e^{t\Delta} u_0 \to u_0$ in $L^p(\sigma)$ as $t \searrow 0$ and $e^{t\Delta} u_0 \in BC([0,\infty); L^p(\sigma))$.
The former combining with \eqref{loc Lp} means that $\Phi[u] \to u_0$ in $L^p(\sigma)$ as $t \searrow 0$.
Hence, it suffices to show that $D[u] \in BC((0,T);L^p(\sigma))$.
It is not hard to see the right continuity.
Indeed, it holds that
\begin{align*}
\left\|D[u](t+\vep) - D[u](t) \right\|_{L^p(\sigma)} & \le \left\| \int_0^t e^{(t-s)\Delta} U(s) - e^{(t+\vep-s)\Delta} U(s) ds \right\|_{L^p(\sigma)} + \left\| \int_t^{t+\vep} e^{(t+\vep-s)\Delta} U(s) ds \right\|_{L^p(\sigma)} \\
& \lesssim W \int_0^t (t-s)^{-\gamma/2} \left\| e^{\vep \Delta} U(s) - U(s) \right\|_{L^r(\sigma V^{-r})} ds + W \vep^{1-\gamma/2} \|u\|_{X_p(T)} \\
& \to 0 \ \tn{as}\ \vep \searrow 0,
\end{align*}
where the convergence of the first term follows from Lebesgue's convergence theorem and $U(s) \in L^r(\sigma V^{-r})$.
To see the left continuity, we divide the difference as follows
\[
\left\|D[u](t) - D[u](t-\vep) \right\|_{L^p(\sigma)} \le \left\| \int_0^{t-\vep} e^{(t-s)\Delta} U(s) - e^{(t-\vep-s)\Delta} U(s) ds \right\|_{L^p(\sigma)} + \left\| \int_{t-\vep}^t e^{(t-s)\Delta} U(s) ds \right\|_{L^p(\sigma)}.
\]
This second term converges to $0$ as $\vep \searrow 0$ from the smoothing estimates; Theorem \ref{main 1}.
By a change of variables, we have that the first term can be estimated by a constant multiplied by
\[
W \int_0^{t-\vep} s^{-\gamma/2} \left\| e^{\vep \Delta} U(t-\vep-s) - U(t-\vep-s) \right\|_{L^r(\sigma V^{-r})} ds.
\]
The integrand is controlled as follows;
\begin{align*}
& \left\| e^{\vep \Delta} U(t-\vep-s) - U(t-\vep-s) \right\|_{L^r(\sigma V^{-r})} \\
& \le \left\| e^{\vep \Delta} U(t-\vep-s) - e^{\vep \Delta} U(t-s) \right\|_{L^r(\sigma V^{-r})} + \left\| e^{\vep \Delta}U(t-s) - U(t-s) \right\|_{L^r(\sigma V^{-r})} \\
& + \left\| U(t-s) - U(t-\vep-s) \right\|_{L^r(\sigma V^{-r})} \\
& \lesssim \left\| U(t-s) - U(t-\vep-s) \right\|_{L^r(\sigma V^{-r})} + \left\| e^{\vep \Delta}U(t-s) - U(t-s) \right\|_{L^r(\sigma V^{-r})}.
\end{align*}
Therefore, we see that
\begin{align*}
\lim_{\vep \searrow 0} & \left\|D[u](t) - D[u](t-\vep) \right\|_{L^p(\sigma)} \lesssim W \limsup_{\vep \searrow 0} \int_0^{t-\vep} s^{-\gamma/2} \left\| U(t-s) - U(t-\vep-s) \right\|_{L^r(\sigma V^{-r})} ds \\
& \lesssim W \limsup_{\vep \searrow 0} \int_0^{t-\vep} s^{-\gamma/2} \left\||u(t-s)|^{\tau-1} u(t-s) - |u(t-\vep-s)|^{\tau-1} u(t-\vep-s) \right\|_{L^r(\sigma)} ds \\
& \lesssim W \|u\|_{X_p(t)}^{\tau-1} \limsup_{\vep \searrow 0} \int_0^t s^{-\gamma/2} \left\| u(t-s) - u(t-\vep-s) \right\|_{L^p(\sigma)} ds \\
& =0,
\end{align*}
where we have used the continuity of $u \in BC([0,T); L^p(\sigma))$ in the last equality.

\smallskip

Thus, the mapping $\Phi$ maps $BC([0,T);L^p(\sigma))$ to itself for all $0<T<\infty$.

\smallskip

\underline{Step 3:} For any $T \in (0,\infty)$ and $u,v \in BC([0,T);L^p(\sigma))$,
\[
\|D[u] - D[v]\|_{X_p(T)} \lesssim W T^{1-\gamma/2} \left(\|u\|_{X_p(T)} + \|v\|_{X_p(T)} \right)^{\tau-1} \|u-v\|_{X_p(T)}.
\]

\smallskip

This is done by a similar argument as follows;
\begin{align*}
\left\| D[u](t) - D[v](t) \right\|_{L^p(\sigma)} & \lesssim W \int_0^t (t-s)^{-\gamma/2} \left\| u(s)\tau - v(s)^\tau \right\|_{L^{p/\tau}(\sigma)} ds \\
& \lesssim W \int_0^t (t-s)^{-\gamma/2} \left\| \left(|u(s)|^{\tau-1} + |v(s)|^{\tau-1} \right) \left|u(s) - v(s) \right| \right\|_{L^{p/\tau}(\sigma)} ds \\
& \lesssim W t^{1-\gamma/2} \left(\|u\|_{X_p(t)}^{\tau-1} + \|v\|_{X_p(t)}^{\tau-1} \right) \|u-v\|_{X_p(t)}.
\end{align*}

\smallskip

Therefore, by taking a sufficiently small $T >0$, we see that $\Phi$ is a contraction mapping on $BC([0,T); L^p(\sigma))$.
Hence, we can find a unique solution $u \in BC([0,T);L^p(\sigma))$ fulfilling $u(t) \to u_0$ in $L^p(\sigma)$ as $t \searrow 0$.
\begin{flushright}
    $\square$
\end{flushright}

\bigskip

\subsection{Proof of Theorem \ref{global}: Global existence}
Let us denote for $0<T \le \infty$,
\[
\|u\|_{X_\infty(T)} := \sup_{0<t<T} s^{(\alpha+n/q)/2} \|u(t)\|_{L^\infty},
\]
$\|u\|_{X(T)} := \|u\|_{X_p(T)} + \|u\|_{X_\infty(T)}$ and
\[
X(T) := \left\{u \in BC([0,T); L^p(\sigma)); \|u\|_{X(T)} < \infty \ \tn{and} \ \lim_{t \searrow 0} t^{(\gamma + n/q)/2} \|u(t)\|_{L^\infty} =0 \right\}.
\]
If $u \in X(T)$, then $\lim_{t \searrow 0} \|u\|_{X_\infty(t)} =0$.

From (ii) in Theorem \ref{main 1} and the proof of Theorem \ref{local} above, we have that
\begin{equation} \label{global linear}
\left\| e^{\cdot \Delta}u_0 \right\|_{X(T)} \lesssim (1+W_2) \|u_0\|_{X(T)}.
\end{equation}
Let $u \in X(\infty)$.
We divide the proof into three steps.

\smallskip

\underline{Step 1:} For any $0<T \le \infty$, it holds that
\[
\|D[u]\|_{X(T)} \lesssim (W_1+W_3) \|u\|_{X_p(T)} \|u\|_{X_\infty(T)}^{\tau-1} \to 0\ \tn{as}\ T \searrow 0,
\]
with the implicit constant independent $T$.
Combining with the linear estimate \eqref{global linear}, this estimate means that $\|\Phi[u]\|_{X(\infty)} \lesssim (1+W_2) \|u_0\|_{L^p(\sigma)} + (W_1+W_3) \|u\|_{X_p(\infty)} \|u\|_{X_\infty(\infty)}^{\tau-1}$.

\smallskip

Our assumption yields that if $t<T$, then it holds
\begin{align*}
\|D[u](t)\|_{L^p(\sigma)} & \lesssim W_1 \int_0^t (t-s)^{-\beta/2} \|U(s)\|_{L^p(\sigma V^{-p})} ds \\
& \lesssim W_1 \int_0^t (t-s)^{-\beta/2} \|u(s)\|_{L^{p\tau}(\sigma)}^\tau ds \\
& \le W_1 \int_0^t (t-s)^{-\beta/2} s^{-(\alpha + n/q)(\tau-1)/2} ds \|u\|_{X_p(T)} \|u\|_{X_\infty(T)}^{\tau-1} \\
& \lesssim W_1 \|u\|_{X_p(T)} \|u\|_{X_\infty(T)}^{\tau-1}.
\end{align*}
On the other hand, one can observe that
\begin{align*}
\left\| D[u](t) \right\|_{L^\infty} & \lesssim W_3 \int_0^t (t-s)^{-(\gamma+n/q_1)/2} \left\| U(s) \right\|_{L^p(\sigma V^{-p})} ds \\
& = W_3 \int_0^t (t-s)^{-(\gamma+n/q_1)/2} \left\| u(s) \right\|_{L^{p\tau}(\sigma)}^\tau ds \\
& \lesssim W_3 \int_0^t (t-s)^{-(\beta_1+n/{q_1})} s^{-(\alpha + n/q)(tau-1)} ds \|u\|_{X_p(T)} \|u\|_{X_\infty(T)}^{\tau-1} \\
& \lesssim W_3 t^{-(\alpha + n/q)} \|u\|_{X_p(T)} \|u\|_{X_\infty(T)}^{\tau-1}.
\end{align*}
The proof of Step 1 is completed.

\smallskip

\underline{Step 2:}
\[
\Phi[u] \in BC([0,\infty); L^p(\sigma))
\]
satisfying $\Phi[u](t) \to u_0$ in $L^p(\sigma)$ as $t \searrow 0$.

\smallskip

Since $e^{t\Delta}u_0 \to u_0$ in $L^p(\sigma)$ as $t \searrow 0$, we see that $\Phi[u]$ is continuous in $L^p(\sigma)$ at $t=0$.
We shall show the continuity on $(0,\infty)$.
Let us take $t \in (0,\infty)$ and $0< \vep < t$.
Following the same argument as that in the proof of Theorem \ref{local}, we obtain the right continuity;
\begin{align*}
\left\|D[u](t+\vep) - D[u](t) \right\|_{L^p(\sigma)} & \le \left\| \int_0^t e^{(t-s)\Delta} U(s) - e^{(t+\vep-s)\Delta} U(s) ds \right\| + \left\| \int_t^{t+\vep} e^{(t+\vep-s)\Delta} U(s) ds \right\| \\
& \lesssim W_1 \int_0^t (t-s)^{-\beta/2} \left\| e^{\vep \Delta} U(s) - U(s) \right\|_{L^p(\sigma V^{-p})} ds \\
& + W_1 t^{-\beta/2} \vep^{1 - (\alpha + n/q)(\tau-1)/2} \|u\|_{X(\infty)}^\tau \\
& \to 0 \ \tn{as}\ \vep \searrow 0.
\end{align*}
To see the left continuity, we divide the integral term into two parts as above.
\[
\left\|D[u](t) - D[u](t-\vep) \right\|_{L^p(\sigma)} \le \left\| \int_0^{t-\vep} e^{(t-s)\Delta} U(s) - e^{(t-\vep-s)\Delta} U(s) ds \right\| + \left\| \int_{t-\vep}^t e^{(t-s)\Delta} U(s) ds \right\|.
\]
The second term can be dominated by a constant multiplied by
\[
W_1 \vep^{1-\beta/2} t^{-(\alpha + n/q)(\tau-1)/2} \|u\|_{X(\infty)}^\tau \to 0 \ \tn{as}\ \vep \to 0.
\]
The first term is controlled by a constant multiplied by
\begin{align*}
W_1 & \int_0^{t-\vep} s^{-\beta/2} \left\| e^{\vep \Delta} U(t-\vep-s) - U(t-\vep-s) \right\|_{L^p(\sigma V^{-p})} ds \\
& \lesssim W_1 \int_0^{t-\vep} s^{-\beta/2} \left\| U(t-s) - U(t-\vep-s) \right\|_{L^p(\sigma V^{-p})} + s^{-\beta/2} \left\| e^{\vep \Delta}U(t-s) - U(t-s) \right\|_{L^p(\sigma V^{-p})} ds.
\end{align*}
This second term tends to $0$ as $\vep \searrow 0$ from Lebesgue's convergence theorem.
We divide another term again as follows;
\begin{align*}
\lesssim & W_1 \int_0^{t-\vep} s^{-\beta/2} (t-\vep-s)^{-(\alpha + n/q)(\tau-1)/2} \|u(t-s) - u(t-\vep-s)\|_{L^p(\sigma)} ds \|u\|_{X_\infty(\infty)}^{\tau-1} \\
& = W_1 \left(\int_0^{(t-\vep)/2} + \int_{(t-\vep)/2}^{t-\vep} \right) s^{-\beta/2} (t-\vep -s)^{-(\alpha+n/q)(\tau-1)/2} \|u(t-s) - u(t-\vep-s)\|_{L^p(\sigma)} ds \|u\|_{X_\infty(\infty)}^{\tau-1} \\
& \lesssim W_1(t-\vep)^{-(\alpha+n/q)(\tau-1)/2} \int_0^{t/2} s^{-\beta/2} \|u(t-s) - u(t-\vep-s)\|_{L^p(\sigma)} ds \|u\|_{X_\infty(\infty)}^{\tau-1} \\
& + W_1 (t-\vep)^{-\beta/2} \int_0^{t/2} s^{-(\alpha+n/q)(\tau-1)/2} \|u(t-s) - u(t-\vep-s)\|_{L^p(\sigma)} ds \|u\|_{X_\infty(\infty)}^{\tau-1} \\
& \to 0 \ \tn{as}\ \vep \searrow 0.
\end{align*}
The proof of Step 2 is completed.

\smallskip

\underline{Step 3:} If $u,v \in X(\infty)$, then it holds that
\[
\|D[u] - D[v]\|_{X(\infty)} \lesssim (W_1+W_3) \left(\|u\|_{X_\infty(\infty)} + \|v\|_{X_\infty(\infty)} \right)^{\tau-1} \|u-v\|_{X(\infty)}.
\]

\smallskip

This is shown as follows;
\begin{align*}
\|D[u](t) - & D[v](t)\|_{L^p(\sigma)} \lesssim W_1 \int_0^t (t-s)^{\beta/2} \left\|u(s)^\tau - v(s)^\tau \right\|_{L^p(\sigma)} ds \\
& \lesssim W_1 \int_0^t (t-s)^{-\beta/2} s^{-(\alpha+n/q)(\tau-1)/2} \|u(s) - v(s)\|_{L^p(\sigma)} ds \left(\|u\|_{X_\infty(\infty)} +\|v\|_{X_\infty(\infty)} \right)^{\tau-1} \\
& \lesssim W_1 \left(\|u\|_{X_\infty(\infty)} +\|v\|_{X_\infty(\infty)} \right)^{\tau-1} \|u-v\|_{X_p(\infty)}
\end{align*}
and
\begin{align*}
& \|D[u](t) - D[v](t)\|_{L^\infty} \\
& \lesssim W_3 \int_0^t (t-s)^{-(\beta_1+n/{q_1})/2} s^{-(\alpha + n/q)(\tau-1)/2} \|u(s)-v(s)\|_{L^p(\sigma)} ds \left(\|u\|_{X_\infty(\infty)} +\|v\|_{X_\infty(\infty)} \right)^{\tau-1} \\
& \lesssim W_3 t^{-(\alpha + n/q)} \left(\|u\|_{X_\infty(\infty)} +\|v\|_{X_\infty(\infty)} \right)^{\tau-1} \|u-v\|_{X_p(\infty)}.
\end{align*}

\smallskip

Therefore, from the fixed-point theorem, we can find a unique solution in a small closed ball in $X(\infty)$.
\begin{flushright}
    $\square$
\end{flushright}

\vspace{5mm}

\section{Appendix: Proofs of Corollaries}
For the sake of completeness, we give a sketch of the proofs of the corollaries, \ref{cor}, \ref{cor local}, \ref{cor global}.
To do this, we prepare basic lemmas.
The first concerns the behavior of the $L^p(B)$-norm of the weights $|x|^\alpha$ and $\lan x \ran^\alpha$.
Following Grafakos \cite{G}, we refer to the open ball $B(x_0,R)$ as a type I ball if $|x_0| \ge 3R$, and as a type II ball if $|x_0| < 3R$.
Since the following lemma can be shown by his argument, we omit the proof.

\begin{lem} \label{est on ball}
Let $\alpha \in \R$ and $B=B(x_0,R)$.

\tn{(i)} If $B$ is of type I, it holds that $\||\cdot|^\alpha\|_{L^1(B)} \approx |x_0|^\alpha R^{n}$.

\medskip

\tn{(ii)} If $B$ is of type II, we have that $\||\cdot|^\alpha\|_{L^1(B)} \gtrsim R^{\alpha + n}$.
Moreover, if $-n< \alpha < \infty$, it follows $\||\cdot|^\alpha\|_{L^1(B)} \approx R^{\alpha + n}$.

\medskip

\tn{(iii)} If $B$ is of type I, it holds $\|\lan \cdot \ran^\alpha \|_{L^1(B)} \approx \lan x_0 \ran^\alpha R^n$.

\medskip

\tn{(iv)} If $B$ is of type II, it follows that
\[
R^n (1+R)^\alpha \lesssim \left\|\lan \cdot \ran^\alpha \right\|_{L^1(B)} \lesssim 
\begin{cases}
\; R^n \ & \tn{if}\ R \le 1 \\
\; 1+R^{\alpha + n} \ & \tn{if}\ R>1.
\end{cases}
\]

\medskip

\tn{(v)}
If $B$ is of type I, we see that $\||\cdot|^\alpha \lan \cdot \ran^\beta \|_{L^1(B)} \approx |x_0|^\alpha \lan x_0 \ran^\beta R^n$.

\medskip

\tn{(vi)} Let $\alpha >-n$.
If $B$ is of type II, one sees that
\[
\left\||\cdot|^\alpha \lan \cdot \ran^\beta \right\|_{L^1(B)} \lesssim
\begin{cases}
\; R^{\alpha+n}\ & \tn{if}\ R \le 1 \\
\; 1+R^{\alpha+\beta+n} \ & \tn{if}\ R>1.
\end{cases}
\]
\end{lem}

\medskip

By using Lemma \ref{est on ball}, one can observe the following.
It is very well-known that $|x|^\alpha \in A_p$ equivalents $-n < \alpha < n(p-1)$ for $p \in (1,\infty)$.
\begin{lem} \label{non-homo Ap}
Let $1 < p < \infty$ and $\alpha, \beta \in \R$.
If
\[
-n \le \alpha + \beta \le n(p-1) \ \tn{and}\ -n < \beta < n(p-1),
\]
then one can see that $w(x) := \lan x \ran^\alpha |x|^\beta \in A_p$.
\end{lem}

\begin{proof}
From Lemma \ref{est on ball}, 
\begin{align*}
[w]_{A_p} & \approx \sup_{B=B(x_0,R)} |B|^{-p} \left\| \lan \cdot \ran^\alpha |\cdot|^\beta \right\|_{L^1(B)} \left\| \lan \cdot \ran^{\alpha (1-p^\prime)} |\cdot|^{\beta(1-p^\prime)} \right\|_{L^1(B)}^{p-1} \\
& \lesssim 1 + \sup_{\substack{B:\tn{type II} \\ R>1}} \left(R^{\alpha + \beta - n(p-1)} + R^{-(\alpha + \beta)-n} \right) \\
& \lesssim 1.
\end{align*}
\end{proof}

\bigskip

\subsection{Proof of Corollary \ref{cor}}
(i) We check the following three conditions
\[
\begin{cases}
\; \circ \ \sigma(x)^{1-p^\prime} = |x|^{\alpha p(1-p^\prime)} \in A_\infty \\
\; \circ \ w(x) = |x|^{\beta q} \in A_\infty \\
\; \circ \ W = \left[|\cdot|^{\alpha p(1-p^\prime)}, |\cdot|^{\beta q} \right]_{A_{p,q}^\gamma} < \infty.
\end{cases}
\]
The first two conditions are that
\[
-n < \alpha (1-p^\prime)p < \infty \iff -\infty <\alpha < n \left(1-\dfrac{1}{p} \right)
\]
and
\[
-n < \beta q < \infty \iff - \dfrac{n}{q} < \beta < \infty,
\]
respectively.
By using Lemma \ref{est on ball}, we have that
\[
W \lesssim \sup_{|x_0| \ge 3 R} R^{\gamma - \alpha + \beta - n(1/p-1/q)} \left(\dfrac{R}{|x_0|} \right)^{\alpha-\beta} + \sup_{R>0} R^{\gamma - \alpha + \beta - n(1/p-1/q)}.
\]
Therefore, the conditions
\[
\alpha - \beta \ge 0 \ \tn{and} \ \gamma = \alpha - \beta + n \left(\dfrac{1}{p} - \dfrac{1}{q} \right)
\]
ensure $W < \infty$.

\medskip

(ii)
We shall check the following three conditions
\[
\begin{cases}
\; \circ \ \sigma(x)^{1-p^\prime} = \lan x \ran^{\alpha p(1-p^\prime)} \in A_\infty \\
\; \circ \ w(x) = |x|^{\beta q} \in A_\infty \\
\; \circ \ W = \left[\lan \cdot \ran^{\alpha p(1-p^\prime)}, |\cdot|^{\beta q} \right]_{A_{p,q}^\gamma} < \infty.
\end{cases}
\]
The first two conditions are that from Lemma \ref{non-homo Ap}
\[
-n \le \alpha (1-p^\prime)p < \infty \iff -\infty <\alpha \le n \left(1-\dfrac{1}{p} \right)
\]
and
\[
-n < \beta q < \infty \iff - \dfrac{n}{q} < \beta < \infty,
\]
respectively.
By using Lemma \ref{est on ball} and the fact $\alpha \ge 0$, we have that
\begin{align*}
W & \lesssim \sup_{\substack{|x_0| \ge 3R \\ R \le 1}} R^{\gamma - n(1/p-1/q)+\beta} \left(\dfrac{R}{|x_0|} \right)^{-\beta} + \sup_{\substack{|x_0| \ge 3 R \\ R > 1}} R^{\gamma-n(1/p-1/q) - (\alpha - \beta)} \left(\dfrac{R}{|x_0|} \right)^{\alpha-\beta} \\
& + \sup_{R \le 1} R^{\gamma-n(1/p-1/q) + \beta} + \sup_{R > 1} R^{\gamma -n(1/p-1/q)- (\alpha - \beta)}.
\end{align*}
Hence, we need the conditions
\[
\beta \le 0 \le \alpha \ \tn{and} \ n \left(\dfrac{1}{p} - \dfrac{1}{q} \right) - \beta \le \gamma \le n \left(\dfrac{1}{p} - \dfrac{1}{q} \right) + (\alpha - \beta) 
\]
to obtain $W < \infty$.

\medskip

(iii) We verify the following three conditions
\[
\begin{cases}
\; \circ \ \sigma(x)^{1-p^\prime} = |x|^{\alpha p(1-p^\prime)} \in A_\infty \\
\; \circ \ w(x) = \lan x \ran^{\beta q} \in A_\infty \\
\; \circ \ W = \left[|\cdot|^{\alpha p(1-p^\prime)}, \lan \cdot \ran^{\beta q} \right]_{A_{p,q}^\gamma} < \infty.
\end{cases}
\]
The first two conditions are that
\[
-n < \alpha (1-p^\prime)p < \infty \iff -\infty <\alpha < n \left(1-\dfrac{1}{p} \right)
\]
and
\[
-n \le \beta q < \infty \iff - \dfrac{n}{q} \le \beta < \infty,
\]
respectively.
We have from $\beta \le 0$,
\begin{align*}
W & \lesssim \sup_{\substack{|x_0| \ge 3 R \\ R \le 1}} R^{\gamma + \beta -n(1/p-1/q)} + \sup_{\substack{|x_0| \ge 3 R \\ R \ge 1}} R^{\gamma - \alpha + \beta -n(1/p-1/q)} \left(\dfrac{R}{|x_0|} \right)^{\alpha-\beta} \\
& + \sup_{R \le 1} R^{\gamma - \alpha - n(1/p-1/q)} + \sup_{R \ge 1} R^{\gamma - \alpha + \beta - n(1/p-1/q)}.
\end{align*}
Then, our conditions
\[
\beta \le 0 \le \alpha \ \tn{and}\ n \left(\dfrac{1}{p} - \dfrac{1}{q} \right) + \alpha \le \gamma \le n \left(\dfrac{1}{p} - \dfrac{1}{q} \right) + (\alpha - \beta)
\]
mean $W < \infty$.

\medskip

(iv) We shall show the following three conditions
\[
\begin{cases}
\; \circ \ \sigma(x)^{1-p^\prime} = \lan x \ran^{\alpha p(1-p^\prime)} \in A_\infty \\
\; \circ \ w(x) = \lan x \ran^{\beta q} \in A_\infty \\
\; \circ \ W = \left[\lan \cdot \ran^{\alpha p(1-p^\prime)}, \lan \cdot \ran^{\beta q} \right]_{A_{p,q}^\gamma} < \infty.
\end{cases}
\]
The first two conditions are that
\[
-n \le \alpha (1-p^\prime)p < \infty \iff -\infty <\alpha \le n \left(1-\dfrac{1}{p} \right)
\]
and
\[
-n \le \beta q < \infty \iff - \dfrac{n}{q} \le \beta < \infty,
\]
respectively.
We have from $\alpha - \beta \ge 0$,
\begin{align*}
W & \lesssim \sup_{\substack{|x_0| \ge 3R \\ R \le 1}} R^{\gamma - n(1/p-1/q)} + \sup_{\substack{|x_0| \ge 3 R \\ R \ge 1}} R^{\gamma - \alpha + \beta - n(1/p-1/q)} \left(\dfrac{R}{|x_0|} \right)^{\alpha-\beta} \\
& + \sup_{R \le 1} R^{\gamma -n(1/p-1/q)} + \sup_{R \ge 1} R^{\gamma - \alpha + \beta - n(1/p-1/q)}.
\end{align*}
One sees $W < \infty$ from the assumptions
\[
\beta \le \alpha \ \tn{and}\ n \left(\dfrac{1}{p} - \dfrac{1}{q} \right) \le \gamma \le n \left(\dfrac{1}{p} - \dfrac{1}{q} \right) + (\alpha - \beta).
\]
\begin{flushright}
    $\square$
\end{flushright}

\bigskip

\subsection{Proof of Corollary \ref{cor local}}
We have to check the following three conditions
\[
\begin{cases}
\; \circ \ \sigma(x) = |x|^\alpha \in A_p \\
\; \circ \ \sigma(x) V(x)^{-r} = |x|^{\alpha - \gamma r} \in A_r,\ \tn{where}\ r:=\dfrac{p}{\tau} \ge 1, \\
\; \circ \ W = \left[\sigma^{1-r^\prime} V^{r^\prime}, \sigma \right]_{A_{r,p}^\beta} < \infty,
\end{cases}
\]
with some $\beta$ satisfying $n(\tau-1)/p \le \beta \le n(1-1/p)$ and $\beta < 2$.
The first two conditions are equivalent to
\[
-n < \alpha < n(p-1) \ \tn{and}\ \dfrac{n+\alpha}{p}\tau - n < \gamma < \dfrac{n+\alpha}{p} \tau.
\]
By using Lemma \ref{est on ball}, we can observe that
\[
W \lesssim \sup_{B: \tn{type I}} R^{\beta - (n+\alpha)(\tau-1)/p + \gamma} \left(\dfrac{R}{|x_0|} \right)^{\alpha(1/r-1/p)-\gamma} + \sup_{B:\tn{type II}} R^{\beta - (n+\alpha)(\tau-1)/p + \gamma}.
\]
Hence, if we denote $\beta := (n+\alpha)(\tau-1)/p - \gamma$, it follows from the assumptions that $\beta$ satisfies the desired conditions, and then $W \lesssim 1$.

\medskip

(ii) We shall check the following three conditions
\[
\begin{cases}
\; \circ \ \sigma(x) = \lan x \ran^\alpha \in A_p \\
\; \circ \ \sigma(x) V(x)^{-r} = \lan x \ran^\alpha|x|^{- \gamma r} \in A_r,\ \tn{where}\ r:=\dfrac{p}{\tau} \ge 1, \\
\; \circ \ W = \left[\lan \cdot \ran^{\alpha(1-r^\prime)} |\cdot|^{\gamma r^\prime}, \lan \cdot \ran^\alpha \right]_{A_{r,p}^\beta} < \infty,
\end{cases}
\]
with some $\beta$ satisfying $n(\tau-1)/p \le \beta \le n(1-1/p)$ and $\beta < 2$.
The first two conditions are equivalent to
\[
-n \le \alpha \le n(p-1),
\]
and
\[
\dfrac{n+\alpha}{p} \tau -n \le \gamma \le \dfrac{n+\alpha}{p}\tau \ \tn{and}\ n \left(1-\dfrac{\tau}{p} \right) -n < \gamma < n \left(1-\dfrac{\tau}{p} \right),
\]
respectively.
By using Lemma \ref{est on ball} and the fact $0 \le \alpha$, one obtains that
\begin{align*}
W & \lesssim \sup_{\substack{B: \tn{type I}\\ R \le 1}} R^{\beta - n(\tau-1)/p + \gamma} \left(\dfrac{R}{|x_0|} \right)^{-\gamma} + \sup_{\substack{B:\tn{type I} \\ R > 1}} R^{\beta - (n+\alpha)(\tau-1)/p + \gamma} \left(\dfrac{R}{|x_0|} \right)^{\alpha(\tau-1)/p - \gamma} \\
& + \sup_{R \le 1} R^{\beta -n(\tau-1)/p+\gamma} + \sup_{R>1} \left(R^{\beta - n} + R^{\beta - \alpha(\tau-1)/p -n + \gamma} \right).
\end{align*}
Our assumptions guarantee that $W \lesssim 1$.

\medskip

(iii) We verify the following three conditions
\[
\begin{cases}
\; \circ \ \sigma(x) = |x|^\alpha \in A_p \\
\; \circ \ \sigma(x) V(x)^{-r} = |x|^\alpha \lan x \ran|^{- \gamma r} \in A_r,\ \tn{where}\ r:=\dfrac{p}{\tau} \ge 1, \\
\; \circ \ W = \left[| \cdot |^{\alpha(1-r^\prime)} \lan \cdot \ran^{\gamma r^\prime}, |\cdot|^\alpha \right]_{A_{r,p}^\beta} < \infty,
\end{cases}
\]
with some $\beta$ satisfying $n(\tau-1)/p \le \beta \le n(1-1/p)$ and $\beta < 2$.
The first two conditions are equivalent to
\[
-n \le \alpha \le n(p-1),
\]
and
\[
 -n \le \alpha \le n \left(\dfrac{p}{\tau} - 1 \right) \ \tn{and}\ \dfrac{n+\alpha}{p}\tau - n \le \gamma \le \dfrac{n+\alpha}{p} \tau,
\]
respectively.
By using Lemma \ref{est on ball} and the fact $\gamma \le 0$, one obtains that
\begin{align*}
W & \lesssim \sup_{\substack{B: \tn{type I}\\ R \le 1}} R^{\beta - (n+\alpha)(\tau-1)/p} \left(\dfrac{R}{|x_0|} \right)^{\alpha(\tau-1)/p} + \sup_{\substack{B:\tn{type I} \\ R > 1}} R^{\beta - (n+\alpha)(\tau-1)/p + \gamma} \left(\dfrac{R}{|x_0|} \right)^{\alpha(\tau-1)/p - \gamma} \\
& + \sup_{R \le 1} R^{\beta - (n+\alpha)(\tau-1)/p} + \sup_{R>1} \left(R^{\beta + (n+\alpha)/p - n} + R^{\beta - (n+\alpha)\tau/p + \gamma} \right).
\end{align*}
The assumptions ensure that $W \lesssim 1$.

\medskip

(iv) We show the following three conditions
\[
\begin{cases}
\; \circ \ \sigma(x) = \lan x \ran^\alpha \in A_p \\
\; \circ \ \sigma(x) V(x)^{-r} = \lan x \ran^{\alpha - \gamma r} \in A_r,\ \tn{where}\ r:=\dfrac{p}{\tau} \ge 1, \\
\; \circ \ W = \left[\lan \cdot \ran^{\alpha(1-r^\prime) + \gamma r^\prime} , \lan \cdot \ran^\alpha \right]_{A_{r,p}^\beta} < \infty,
\end{cases}
\]
with some $\beta$ satisfying $n(\tau-1)/p \le \beta \le n(1-1/p)$ and $\beta < 2$.
The first two conditions are equivalent to
\[
-n \le \alpha \le n(p-1) \ \tn{and}\ \dfrac{n+\alpha}{p} \tau - n \le \gamma \le \dfrac{n+\alpha}{p} \tau,
\]
respectively.
By using Lemma \ref{est on ball} and the fact $\gamma \le 0$, one obtains that
\[
W \lesssim \sup_{B: \tn{type I}} R^{\beta - (n+\alpha)(\tau-1)/p + \gamma} \left(\dfrac{R}{\lan x_0 \ran} \right)^{\alpha(\tau-1)/p - \gamma} + \sup_{R \le 1} R^{\beta - n(\tau-1)/p} + \sup_{R>1} \left(R^{\beta-n} + R^{\beta -(n+\alpha)(\tau-1)/p + \gamma} \right).
\]
The assumptions ensure that $W \lesssim 1$.
\begin{flushright}
    $\square$
\end{flushright}

\bigskip

\subsection{Proof of Corollary \ref{cor global}}
(i) We verify the following five conditions
\[
\begin{cases}
\; \circ \ \sigma(x) = |x|^\alpha \in A_p, \\
\; \circ \ \sigma(x) V(x)^{-p} = |x|^{\alpha-\gamma p} \in A_p, \\
\; \circ \ W_1= \left[|\cdot|^{\alpha (1-p^\prime) + \gamma p^\prime},\ |\cdot|^{\alpha} \right]_{A^\theta_{p,p}} < \infty, \\
\circ \ W_2= \left\| |\cdot|^{-\alpha/p} \right\|_{\calM^{n/\ell}_{p^\prime}} < \infty,\\
\; \circ \ W_3 = \left\| |\cdot|^{-\alpha/p+\gamma} \right\|_{\calM^{n/{\ell_1}}_{p^\prime}} < \infty.
\end{cases}
\]
The first conditions are equivalent to
\[
-n < \alpha < n(p-1) \ \tn{and} \ \dfrac{n+\alpha}{p} -n < \gamma< \dfrac{n+\alpha}{p}.
\]
We can see that
\[
W_1 \lesssim \sup_{|x_0| \ge 3 R} R^{\theta + \gamma} \left(\dfrac{R}{|x_0|} \right)^{-\gamma} + \sup_{R>0} R^{\theta+\gamma},
\]
\[
W_2 \lesssim \sup_{|x_0| \ge 3 R} R^{\ell - \alpha/p} \left(\dfrac{R}{|x_0|} \right)^{\alpha/p} + \sup_{R>0} R^{\ell-\alpha/p}
\]
and
\[
W_3 \lesssim \sup_{|x_0| \ge 3R} R^{\ell_1 - \alpha/p+\gamma} \left(\dfrac{R}{|x_0|} \right)^{\alpha/p - \gamma} + \sup_{R>0} R^{\ell_1-\alpha/p+\gamma}.
\]
Our conditions
\[
\gamma \le 0 \le \alpha \ \tn{and}\ \beta = \dfrac{2+\gamma}{\tau-1} - \dfrac{n}{q} = n \left(\dfrac{1}{p} - \dfrac{1}{q} \right) + \dfrac{\alpha}{p},
\]
are sufficient for ensuring that for  $W_1+W_2<\infty$.
In this case, we have
\[
\gamma = \dfrac{n+\alpha}{p}(\tau-1) -2.
\]

\medskip

(ii) We shall verify the following five conditions
\[
\begin{cases}
\; \circ \ \sigma(x)= \lan x \ran^\alpha \in A_p, \\
\; \circ \ \sigma(x) V(x)^{-p} = \lan x \ran^\alpha |x|^{-\gamma p} \in A_p, \\
\; \circ \ W_1= \left[ \lan \cdot \ran^{\alpha (1-p^\prime)} |\cdot|^{\gamma p^\prime},\ \lan \cdot \ran^{\alpha} \right]_{A^\theta_{p,p}} < \infty, \\
\circ \ W_2= \left\| \lan \cdot \ran^{-\alpha/p} \right\|_{\calM^{n/\ell}_{p^\prime}} < \infty, \\
\; \circ \ W_3 = \left\| \lan \cdot \ran^{-\alpha/p} |\cdot|^{\gamma} \right\|_{\calM^{n/{\ell_1}}_{p^\prime}} < \infty.
\end{cases}
\]
The first two conditions are equivalent to
\[
-n \le \alpha \le n(p-1),\ \dfrac{n+\alpha}{p} -n \le \gamma \le \dfrac{n+\alpha}{p} \ \tn{and}\ \dfrac{n}{p} -n < \gamma < \dfrac{n}{p}.
\]
We have that
\[
W_1 \lesssim \sup_{B: \tn{type I}} R^{\theta + \gamma} \left(\dfrac{R}{|x_0|} \right)^{-\gamma} + \sup_{R \le 1} R^{\theta + \gamma} + \sup_{R \ge 1} \left(R^{\theta+\gamma} + R^{\theta-n + (n+\alpha)/p} \right),
\]
\begin{align*}
W_2 \lesssim \sup_{\substack{B:\tn{type I} \\ R \le 1}} R^\ell + \sup_{\substack{B:\tn{type I} \\ R \ge 1}} R^{\ell -\alpha/p} \left(\dfrac{R}{|x_0|} \right)^{\alpha/p} + \sup_{R \le 1} R^\ell + \sup_{R \ge 1} \left(R^{\ell - n(1-1/p)} + R^{\ell - \alpha/p} \right)
\end{align*}
and
\begin{align*}
W_3 & \lesssim \sup_{\substack{B:\tn{type I} \\ R \le 1}} R^{\ell_1 + \gamma} \left(\dfrac{R}{|x_0|} \right)^{-\gamma} + \sup_{\substack{B:\tn{type I} \\ R \ge 1}} R^{\ell_1 -\alpha/p + \gamma} \left(\dfrac{R}{|x_0|} \right)^{\alpha/p-\gamma} \\
& + \sup_{R \le 1} R^{\ell_1 + \gamma } + \sup_{R \ge 1} \left(R^{\ell_1 - n(1-1/p)} + R^{\ell_1 - \alpha/p + \gamma} \right).
\end{align*}
Here, we have used the fact that $0 \le \alpha$ for the second and third estimates. 
To denote
\[
\beta = \dfrac{2+\gamma}{\tau-1} - \dfrac{n}{q}.
\]
is sufficient to ensure that $W_1 < \infty$.

\medskip

(iii) We shall show the following conditions
\[
\begin{cases}
\; \circ \ \sigma(x) = |x|^\alpha \in A_p, \\
\; \circ \ \sigma(x) V(x)^{-p} = |x|^\alpha \lan x \ran^{-\gamma p} \in A_p, \\
\; \circ \ W_1= \left[|\cdot|^{\alpha (1-p^\prime)} \lan \cdot \ran^{\gamma p^\prime},\ |\cdot|^{\alpha} \right]_{A^\theta_{p,p}} < \infty, \\
\circ \ W_2= \left\| |\cdot|^{-\alpha/p} \right\|_{\calM^{n/\ell}_{p^\prime}} < \infty,\\
\; \circ \ W_3 = \left\| |\cdot|^{-\alpha/p} \lan \cdot \ran^\gamma \right\|_{\calM^{n/{\ell_1}}_{p^\prime}} < \infty.
\end{cases}
\]
The first conditions are equivalent to
\[
-n < \alpha < n(p-1),\ \dfrac{n+\alpha}{p} -n \le \gamma \le \dfrac{n+\alpha}{p} \ \tn{and} \ \dfrac{n}{p} -n < \gamma< \dfrac{n}{p}.
\]
From Lemma \ref{est on ball}, one obtains that
\[
W_1 \lesssim \sup_{\substack{B:\tn{type I} \\ R \le 1}} R^\theta + \sup_{\substack{B:\tn{type I} \\ R \ge 1}} R^{\theta + \gamma} \left(\dfrac{R}{|x_0|} \right)^{-\gamma} + \sup_{R \le 1} R^\theta + \sup_{R \ge 1} \left(R^{\theta + (n+\alpha)/p -n} + R^{\theta + \gamma} \right),
\]
\[
W_2 \lesssim \sup_{B: \tn{type I}} R^{\ell - \alpha/p} \left(\dfrac{R}{|x_0|} \right)^{\alpha/p} + \sup_{R>0} R^{\ell - \alpha/p}
\]
and
\begin{align*}
W_3 & \lesssim \sup_{\substack{B:\tn{type I} \\ R \le 1}} R^{\ell_1-\alpha/p} \left(\dfrac{R}{|x_0|} \right)^{\alpha/p} + \sup_{\substack{B:\tn{type I} \\ R \ge 1}} R^{\ell_1-\alpha/p+\gamma} \left(\dfrac{R}{|x_0|} \right)^{\alpha/p-\gamma} \\
& + \sup_{R \le 1} R^{\ell_1-\alpha/p} + \sup_{R \ge 1} \left(R^{\ell_1-n(1-1/p)} + R^{\ell_1-\alpha/p+\gamma} \right).
\end{align*}
Here, we have used the fact $\gamma \le 0$ for the estimates of $W_1$ and $W_3$.
To denote
\[
0 \le \alpha \ \tn{and}\ \beta = \dfrac{2-\alpha/p}{\tau-1} - \dfrac{n}{q}.
\]
is sufficient to ensure that $W_2 < \infty$.

\medskip

(iv) Our task is to verify the following five conditions
\[
\begin{cases}
\; \circ \ \sigma(x) = \lan x \ran^\alpha \in A_p, \\
\; \circ \ \sigma(x) V(x)^{-p} = \lan x \ran^{\alpha-\gamma p} \in A_p, \\
\; \circ \ W_1= \left[\lan \cdot \ran^{\alpha (1-p^\prime) + \gamma p^\prime},\ \lan \cdot \ran^{\alpha} \right]_{A^\theta_{p,p}} < \infty, \\
\circ \ W_2= \left\| \lan \cdot \ran^{-\alpha/p} \right\|_{\calM^{n/\ell}_{p^\prime}} < \infty,\\
\; \circ \ W_3 = \left\| \lan \cdot \ran^{-\alpha/p+\gamma} \right\|_{\calM^{n/{\ell_1}}_{p^\prime}} < \infty.
\end{cases}
\]
The first conditions are equivalent to
\[
-n \le \alpha \le n(p-1) \ \tn{and} \ \dfrac{n+\alpha}{p} -n \le \gamma \le \dfrac{n+\alpha}{p}.
\]
Lemma \ref{est on ball} yields that
\[
W_1 \lesssim \sup_{\substack{B: \tn{type I} \\ R \le 1}} R^\theta + \sup_{\substack{B: \tn{type I} \\ R \ge 1}} R^{\theta + \gamma} \left(\dfrac{R}{|x_0|} \right)^{-\gamma} + \sup_{R \le 1} R^\theta + \sup_{R \ge 1} \left(R^{\theta - n + (n+\alpha)/p} + R^{\theta + \gamma} \right),
\]
\[
W_2 \lesssim \sup_{\substack{B: \tn{type I} \\ R \le 1}} R^\ell + \sup_{\substack{B: \tn{type I} \\ R \ge 1}} R^{\ell - \alpha/p} \left(\dfrac{R}{|x_0|} \right)^{\alpha/p} + \sup_{R \le 1} R^\ell + \sup_{R \ge 1} \left(R^{\ell - n(1-1/p)} + R^{\ell - \alpha/p} \right)
\]
and
\[
W_3 \lesssim \sup_{\substack{B: \tn{type I} \\ R \le 1}} R^{\ell_1} + \sup_{\substack{B: \tn{type I} \\ R \ge 1}} R^{\ell_1 - \alpha/p + \gamma} \left(\dfrac{R}{|x_0|} \right)^{\alpha/p-\gamma} + \sup_{R \le 1} R^{\ell_1} + \sup_{R \ge 1} \left(R^{\ell_1 -n(1-1/p)} + R^{\ell_1 - \alpha/p + \gamma} \right).
\]
We have used the fact $\gamma \le 0$ in the estimate of $W_1$, the fact $0 \le \alpha$ in the estimate of $W_2$, and the fact $\gamma \le 0 \le \alpha$ in the estimate of $W_3$.
The condition $W_1+W_2+W_3 < \infty$ is verified from the condition on $\beta$;
\[
\max \left(\dfrac{2+\gamma}{\tau-1}, \dfrac{n}{p}, \dfrac{2+\gamma-n/p}{\tau-2} \right) \le \beta \le \min \left(\dfrac{2}{\tau-1}, \dfrac{n+\alpha}{p}, \dfrac{2-n/p}{\tau-2} \right).
\]
The existence of such $\beta$ is guaranteed by our assumptions.
\begin{flushright}
    $\square$
\end{flushright}

\vspace{5mm}

\subsection*{Acknowledgements}
This work was supported by JSPS KAKENHI Grant Numbers 18KK0072, 20H01815, 23K03181.

\end{document}